\documentclass{amsart}

\usepackage{amsmath}
\usepackage{amsthm}
\usepackage{amsfonts}
\usepackage{amssymb}
\usepackage{enumerate}
\usepackage{mathrsfs}
\usepackage[nocompress]{cite}
\usepackage{graphicx}

\let\Re\undefined
\let\Im\undefined
\DeclareMathOperator{\Re}{\mathscr{R}}
\DeclareMathOperator{\Im}{\mathscr{I}}

\DeclareMathOperator{\supp}{supp}

\newcommand{\E}{\mathbb{E}}
\newcommand{\C}{\mathbb{C}}
\renewcommand{\P}{\mathbb{P}}
\renewcommand\Re{\operatorname{Re}}
\renewcommand\Im{\operatorname{Im}}
\newcommand{\eps}{\varepsilon}

\renewcommand{\d}{\, d }

\def\Q{\mathbb{Q}}

\newcommand{\lp}{\log_+}
\newcommand{\lm}{\log_-}

\theoremstyle{plain}
  \newtheorem{theorem}{Theorem}[section]
  
  \newtheorem{lemma}[theorem]{Lemma}

  \newtheorem{proposition}[theorem]{Proposition}

\theoremstyle{definition}
  \newtheorem{definition}[theorem]{Definition}
  \newtheorem{example}[theorem]{Example}

\theoremstyle{remark}
  \newtheorem{remark}[theorem]{Remark}

\begin{document}

\title{Sums of random polynomials with differing degrees}

\author{Isabelle Kraus}
\address{Department of Mathematics\\ University of Colorado\\ Campus Box 395\\ Boulder, CO 80309-0395\\USA}
\email{isabelle.kraus@colorado.edu}

\author{Marcus Michelen}
\address{Department of Mathematics, Statistics, and Computer Science\\ University of Illinois at Chicago}
\email{michelen.math@gmail.com}
\thanks{M. Michelen has been supported in part by NSF grants DMS-2137623 and DMS-2246624.}

\author{Sean O'Rourke}
\address{Department of Mathematics\\ University of Colorado\\ Campus Box 395\\ Boulder, CO 80309-0395\\USA}
\email{sean.d.orourke@colorado.edu}
\thanks{S. O'Rourke has been supported in part by NSF grant DMS-1810500.}

\begin{abstract}
Let $\mu$ and $\nu$ be probability measures in the complex plane, and let $p$ and $q$ be independent random polynomials of degree $n$, whose roots are chosen independently from $\mu$ and $\nu$, respectively.  Under assumptions on the measures $\mu$ and $\nu$, the limiting distribution for the zeros of the sum $p+q$ was by computed by Reddy and the third author [J. Math. Anal. Appl. 495 (2021) 124719] as $n \to \infty$.  In this paper, we generalize and extend this result to the case where $p$ and $q$ have different degrees.   In this case, the logarithmic potential of the limiting distribution is given by the pointwise maximum of the logarithmic potentials of $\mu$ and $\nu$, scaled by the limiting ratio of the degrees of $p$ and $q$.  Additionally, our approach provides a complete description of the limiting distribution for the zeros of $p + q$ for any pair of measures $\mu$ and $\nu$, with different limiting behavior shown in the case when at least one of the measures fails to have a logarithmic moment.  
\end{abstract}

\maketitle

\section{Introduction} 

Given monic polynomials $p$ and $q$, what can be said about the roots of $p + q$?  While this question has been explored previously \cite{MR1188001,MR3196447,math/0612833,MR1923392,MR1908370,Marden,MR1911767,MR202981,Walsh,MR171902}, especially when $p$ and $q$ are deterministic, the goal of this paper is to settle this question for certain classes of random polynomials.

Motivated by the results of Reddy and the third author \cite{ORourke}, this paper focuses on a model of random polynomials with independent and identically distributed (iid) roots.  Namely, we consider monic polynomials (in a single complex variable) of the form 
\[ p_n(z) := \prod_{i=1}^n (z - X_i), \]
where the roots $X_1, \ldots, X_n$ are iid random variables in the complex plane.  Various properties of $p_n$ have been studied by a number of authors, see \cite{MR4136480,MR3896083,MR3689975,MR3363974,MR3283656,MR3940764,MR3698743,MR2970701,MR3318313,1801.08974} and references therein.

The results in \cite{ORourke} describe zeros of sums of random polynomials in this class.  Specifically, let 
\[ p_n(z) := \prod_{i=1}^n (z - X_i), \qquad q_n(z) := \prod_{j=1}^n (z - Y_j) \]
be two independent random polynomials of degree $n$, whose roots $X_1, \ldots, X_n$ and $Y_1, \ldots, Y_n$ are chosen independently from probability measures $\mu$ and $\nu$ in the complex plane, respectively.  The main results of \cite{ORourke} describe the limiting distribution for the zeros of the sum $p +q$ as $n \to \infty$ in terms of the logarithmic potentials of $\mu$ and $\nu$ under certain assumptions on the measures $\mu$ and $\nu$.  More generally, the results in \cite{ORourke} apply to sums of $m$ independent random polynomials when $m$ is fixed and $n$ tends to infinity.  In order to state these results, we must first introduce some definitions and notation.  

Let $\mathcal{P}(\mathbb{C})$ be the set of probability measures on $\mathbb{C}$. We let $\mathcal{P}_{\lp}(\mathbb{C})$ denote the set of $\mu \in \mathcal{P}(\mathbb{C})$ such that
\[ \int_{\mathbb{C}} \lp |w| \d \mu(w) < \infty, \]
where
\[ \lp x = \begin{cases} 
      0, & 0 \leq x \leq 1, \\
      \log x, & x \geq 1,
   \end{cases} \]
for $x \geq 0$.  
That is, $\mathcal{P}_{\lp}(\mathbb{C})$ consists of the probability measures on $\mathbb{C}$ which integrate $\log |\cdot|$ in a neighborhood of infinity.  
\begin{definition} \label{def:logpot}
The \textit{logarithmic potential} $U_\mu$ of $\mu \in \mathcal{P}_{\lp}(\mathbb{C})$ is the function $U_\mu: \mathbb{C} \to [-\infty, +\infty)$ defined for all $z \in \mathbb{C}$ by 
\[ U_\mu(z) := \int_\mathbb{C} \log |z-w| \ d\mu(w). \]
\end{definition}

For a measure $\mu \in \mathcal{P}(\mathbb{C})$, we let $\supp(\mu) \subset \mathbb{C}$ denote the support of $\mu$; $\mu$ is said to be compactly supported if $\supp(\mu)$ is compact.    
%
%
%
%
Let $\lambda$ denote the Lebesgue measure on $\mathbb{C}$, and let $C_{c}^\infty(\mathbb{C})$ denote the set of all smooth functions $\varphi: \mathbb{C} \to \mathbb{C}$ with compact support.

\begin{theorem}[Theorem 1.10 in \cite{ORourke}] \label{thm:RO}
Let $m \geq 2$ be a fixed integer, and assume $\mu_1, \ldots, \mu_m \in \mathcal{P}(\mathbb{C})$ have compact support.  Assume for each $1 \leq k \leq m-1$, the measure $\mu_k$ is not supported on a circle\footnote{A measure $\mu \in \mathcal{P}(\mathbb{C})$ is said to be supported on a circle if there exists $z_0 \in \mathbb{C}$ and $r \geq 0$ so that $\supp (\mu) \subset \{z \in \mathbb{C} : |z-z_0| = r\}$.}.  Let $\{X_{i,k} : 1 \leq k \leq m, i \geq 1\}$ be a collection of independent random variables so that $X_{i,k}$ has distribution $\mu_k$ for each $i \geq 1$.  For each $n \geq 1$, define the degree $n$ polynomials 
\[ p_{n,k}(z) := \prod_{i=1}^n (z - X_{i,k}), \qquad 1 \leq k \leq m. \]
Then there exists a (deterministic) probability measure $\rho$ on $\mathbb{C}$ so that, for any $\varphi \in C_c^\infty(\mathbb{C})$, 
\[ \frac{1}{n} \sum_{i=1}^n \varphi(z_i^{(n)}) \longrightarrow \int_{\mathbb{C}} \varphi \ d \rho \]
in probability as $n \to \infty$, where $z_1^{(n)}, \ldots, z_n^{(n)}$ are the zeros of the sum $\sum_{k=1}^m p_{n,k}$.  Here, $\rho$ depends only on $\mu_1, \ldots, \mu_m$ and is uniquely defined by the condition that 
\[ \int_{\mathbb{C}} \varphi \ d \rho = \frac{1}{2 \pi} \int_{\mathbb{C}} \Delta \varphi(z) \left( \max_{1 \leq k \leq m} U_{\mu_k}(z) \right) \ d \lambda (z) \quad \text{for all } \varphi \in C_c^\infty(\mathbb{C}). \]
\end{theorem}

\subsection{Contributions of this paper}

Our main results generalize Theorem \ref{thm:RO} in several key ways:
\begin{itemize}
\item We allow the polynomials $p_{n,1}$, \ldots, $p_{n,m}$ to have have different degrees.  That is, we consider polynomials of the form 
\[ p_{n,k}(z) := \prod_{i=1}^{n_k} (z - X_{i,k}), \qquad 1 \leq k \leq m, \]
where $\{n_1\}_{n \geq 1}, \ldots, \{n_m\}_{n \geq 1}$ are sequences of natural numbers, indexed by $n$ (so that $n_k = n_k(n)$).  We assume that $n_1 = n \geq n_k$ for each $2 \leq k \leq m$ and all natural numbers $n$.  Our main result does not require that the sequences $\{n_2\}_{n \geq 1}, \ldots, \{n_m\}_{n \geq 1}$ tend to infinity with $n$; in fact, even when these sequences do not tend to infinity, they can still influence the limiting distribution of the zeros of the sum (see Example \ref{ex:main} below).  
\item Theorem \ref{thm:RO} makes two key assumptions about the measure $\mu_1, \ldots, \mu_m$: it requires the measures be compactly supported and not supported on circles.  These technical assumptions were required due to the proof method used in \cite{ORourke}.  The proof given in this paper is substantially different than that given in \cite{ORourke}, and we do not require any assumptions about the support of the measures.  In particular, we observe a new phenomenon, where the behavior of the roots of the sum depends on how heavy-tailed the measures $\mu_1, \ldots, \mu_m$ are (see Section \ref{sss:heavy}). 
\end{itemize}

\subsection{Main results}
Our main results are divided into two theorems: the first theorem captures the behavior of the zeros of the sum when $\mu_1, \ldots, \mu_m \in \mathcal{P}_{\lp}(\mathbb{C})$ (which we call the \emph{light-tailed case}) and the second case describes a different behavior when $\mu_k \in \mathcal{P}(\mathbb{C}) \setminus \mathcal{P}_{\lp}(\mathbb{C})$ for some $k$ (which we call the \emph{heavy-tailed case}).  We begin with a definition.

\begin{definition}[Weak convergence of (random) probability measures]
Let $\{\rho_n\}_{n \geq 1}$ be a sequence of deterministic probability measures on $\mathbb{C}$, and let $\rho \in \mathcal{P}(\mathbb{C})$ be deterministic.  We say \emph{$\rho_n$ converges weakly to $\rho$} if, for all continuous and bounded functions $\varphi: \mathbb{C} \to \mathbb{C}$, 
\begin{equation} \label{eq:weakconv}
	\int_{\mathbb{C}} \varphi \d \rho_n  \longrightarrow \int_{\mathbb{C}} \varphi \d \rho 
\end{equation}
as $n \to \infty$.   We say a sequence $\{\rho_n\}_{n \geq 1}$ of random probability measures on $\mathbb{C}$ \emph{converges weakly in probability} (respectively \emph{almost surely}) to a deterministic measure $\rho \in \mathcal{P}(\mathbb{C})$ if, for each continuous and bounded function $\varphi: \mathbb{C} \to \mathbb{C}$, the convergence in \eqref{eq:weakconv} holds in probability (respectively almost surely) as $n \to \infty$.    
\end{definition}

\subsubsection{The light-tailed case}
Let $\mu_1, \ldots, \mu_m \in \mathcal{P}_{\lp}(\mathbb{C})$, and let $\{X_{i, k} : 1 \leq k \leq m, i \geq 1\}$ be a collection of independent complex-valued random variables so that $X_{i, k}$ has distribution $\mu_k$ for all $1 \leq k \leq m$ and $i \geq 1$.  
We consider polynomials of the form 
\begin{equation} \label{eq:polys}
	p_{n, k}(z) := \prod_{i=1}^{n_k} (z - X_{i, k} ), \qquad 1 \leq k \leq m, 
\end{equation} 
where $\{n_1\}_{n \geq 1}, \ldots, \{n_m\}_{n \geq 1}$ are sequences of natural numbers, indexed by $n$ (so that $n_k = n_k(n)$).  We assume $n_1 := n \geq n_k$ for all $2 \leq k \leq m$ and for every natural number $n$.  

Our main result below describes the limiting distribution of the zeros of the sum $\sum_{k=1}^m p_{n,k}$.  By the assumptions above, it follows that $\sum_{k=1}^m p_{n,k}$ is always a degree $n$ polynomial and so has $n$ zeros, which we denote as $z_1^{(n)}, \ldots, z_n^{(n)}$.  We let
\begin{equation} \label{eq:rhon}
	\rho_n := \frac{1}{n} \sum_{i=1}^n \delta_{z_i^{(n)}} 
\end{equation} 
be the empirical measure constructed from the zeros of the sum $\sum_{k=1}^m p_{n,k}$; here, $\delta_z$ denotes the  point mass at $z$ in the complex plane.  
Unlike the case when all the polynomials have the same degree (Theorem \ref{thm:RO}), the case when the polynomials have different degrees requires several new parameters.  These new parameters are $c_1, \ldots, c_m \in [0,1]$ and are defined by the limiting ratio of the degrees:
\begin{equation} \label{eq:limits}
	c_k := \lim_{n \to \infty} \frac{n_k}{n}, \qquad 1 \leq k \leq m, 
\end{equation} 
where we assume the limits in \eqref{eq:limits} exist.  In particular, it is always the case that $c_1 = 1$ since $n_1 := n$ for all natural numbers $n$.  As an example, if $m = 3$ and we have $n_1 = n$, $n_2 = \lceil \sqrt{n} \rceil$, and $n_3 = \lceil n/2 \rceil$  for all $n$, then $c_1 = 1$, $c_2 = 0$, and $c_3 = 1/2$.

Our first main result below shows that the limiting distribution for $\rho_n$ depends only on the measures $\mu_1, \ldots, \mu_m$ and the limiting ratios $c_1, \ldots, c_m$.

\begin{theorem}[Main result: light-tailed case] \label{thm:main}
Let $m \geq 2$ be a fixed integer, and assume $\mu_1, \ldots, \mu_m \in \mathcal{P}_{\lp}(\mathbb{C})$.  Let $\{X_{i,k} : 1 \leq k \leq m, i \geq 1\}$ be a collection of independent random variables so that $X_{i,k}$ has distribution $\mu_k$ for each $i \geq 1$.  
For each $n \geq 1$, define the polynomials $p_{n,k}$ as in \eqref{eq:polys}, 
where $\{n_1\}_{n \geq 1}, \ldots, \{n_m\}_{n \geq 1}$ are sequences of natural numbers, indexed by $n$, so that $n_1 := n \geq n_k$ for each $2 \leq k \leq m$ and for all natural numbers $n$.  In addition, assume the limits in \eqref{eq:limits} hold for some $c_1, \ldots, c_m \in [0, 1]$; for each $1 \leq k \leq m$, if the sequence $\{n_k\}_{n \geq 1}$ is unbounded, assume $\mu_k$ is non-degenerate.  For each $n \geq 1$, let $\rho_n$ be the empirical measure defined in \eqref{eq:rhon}, where $z_1^{(n)}, \ldots, z_{n}^{(n)}$ are the zeros of the sum $\sum_{k=1}^m p_{n,k}$. 
Then there exists a (deterministic) probability measure $\rho$ on $\mathbb{C}$ so that $\rho_n$ converges weakly to $\rho$ in probability as $n \to \infty$.  Here, $\rho$ depends only on $\mu_1, \ldots, \mu_m$ and $c_1, \ldots, c_m$ and is uniquely defined by the condition that 
\begin{equation}\label{eq:main}
\int_\mathbb{C} \varphi \ d\rho = \frac{1}{2\pi}\int_\mathbb{C} \Delta \varphi(z)\left( \max_{1 \leq k \leq m} {c_k}U_{\mu_k}(z) \right)\ d\lambda(z) \quad \text{ for all } \varphi \in C_c^\infty(\mathbb{C}),
\end{equation}
where we use the convention that ${c_k}U_{\mu_k}(z) = 0$ for all $z \in \mathbb{C}$ if $c_k = 0$.
\end{theorem}

We make a few remarks concerning Theorem \ref{thm:main}.  First, in the case where $n_1 = n_2 = \cdots = n_m = n$, it follows that $c_1 = c_2 = \cdots = c_m = 1$, and we recover a generalized version of Theorem \ref{thm:RO} which makes no assumptions about the supports of the measures $\mu_1, \ldots, \mu_m$.  Second, the defining relation for $\rho$ given in \eqref{eq:main} implies that the function $U(z) := \max_{1 \leq k \leq m} {c_k}U_{\mu_k}(z)$ is the logarithmic potential of $\rho$. In fact, we can write \eqref{eq:main} as 
\[ \rho = \frac{1}{2 \pi} \Delta U, \]
where the Laplacian is interpreted in the distributional sense (see Section 3.7 in  \cite{Ransford}). 
Third, while it might be tempting to conjecture that only the highest degree polynomials in the sum affect the limiting distribution, as the following example shows, this is not the case.

\begin{example}\label{ex:main}
Consider Theorem \ref{thm:main} in the case where $m = 2$, $\mu_1$ is the uniform probability measure on the unit disk in the complex plane centered at the origin, and $\mu_2$ is the uniform probability measure on the unit disk centered at $2$. Let $n_1 := n$ for each $n \in \mathbb{N}$, and let $n_2$ be any sequence of natural numbers indexed by $n$ chosen so that $\lim_{n \to \infty} \frac{n_2}{n} = 0$ (in particular, this includes the case where $n_2$ is constant). Define the polynomials $p_{n,1}$ and $p_{n,2}$ as in  \eqref{eq:polys}.  

Since the degree of $p_{n,2}$ is significantly smaller than the degree of $p_{n,1}$, it might be natural to conjecture that the limiting distribution for the roots of the sum $p_{n,1} + p_{n,2}$ is given by $\mu_1$.  However, this is not the case (even when $n_2 = 1$ for all $n$).  In this case, the limiting distribution for the zeros of $p_{n,1} + p_{n,2}$ is given by the uniform probability measure on the unit circle centered at the origin.  
To see this, note that for every $z \in \mathbb{C}$, 
$$ U_{\mu_1}(z) =  \begin{cases} 
      \log|z|, & \text{ if } |z| > 1, \\
      \frac{1}{2}(|z|^2-1), & \text{ if } |z| \leq 1, \\
 \end{cases}$$
see for instance \cite{Saff}, and
$$ c_2 U_{\mu_2}(z) = 0$$
since $c_2 = 0.$
It follows that 
$$U(z) := \max\{U_{\mu_1}(z), 0\} =  
\begin{cases} 
      \log|z|, & \text{ if } |z| > 1, \\
      0, & \text{ if } |z| \leq 1. \\
\end{cases}$$
Using the mean value property for harmonic functions, $U$ can be seen to be equal almost everywhere to the logarithmic potential of the uniform probability measure on the unit circle centered at the origin.  Thus, by uniqueness of the logarithmic potential (see, for instance, Lemma 4.1 from \cite{Bordenave}), it follows that the resultant measure $\rho$ given in Theorem \ref{thm:main} is the uniform probability measure on the unit circle centered at the origin.  A numerical simulation of this example is shown in Figure \ref{fig:light_tail}.  
\end{example}

\begin{figure}
    \centering
    \includegraphics[width=0.95\textwidth]{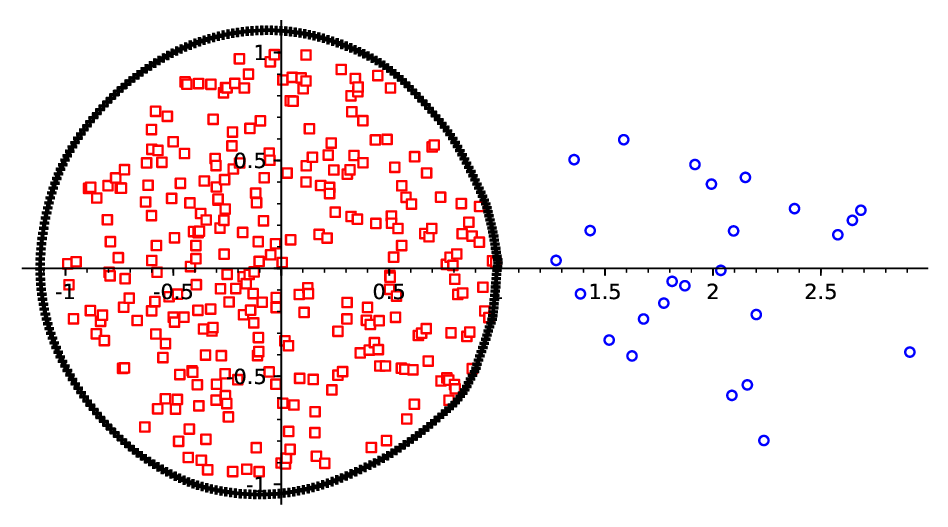}
    \caption{A numerical simulation of Example \ref{ex:main}. The red squares represent the roots of $p_{n,1}$, which are uniform on the unit disk centered at the origin. In this simulation, $p_{n,1}$ has degree $300$. The blue circles represent the roots of $p_{n,2}$, which are uniform on the unit disk centered at $2$. In this simulation, $p_{n,2}$ has degree $25$. The black crosses represent the roots of the sum $p_{n,1} + p_{n,2}$.}
    \label{fig:light_tail}
\end{figure}

\begin{example}
Consider Theorem \ref{thm:main} in the case where $m = 2$, $\mu_1$ is the uniform probability measure on the unit circle in the complex plane centered at the origin, and $\mu_2$ is the uniform probability measure on the circle centered at the origin with radius $r \geq 1$.  Let $n_1 := n$ for each $n \in \mathbb{N}$, and let $n_2$ be any sequence of natural numbers indexed by $n$ chosen so that $\lim_{n \to \infty} \frac{n_2}{n} = c \in [0, 1]$. Define the polynomials $p_{n,1}$ and $p_{n,2}$ as in  \eqref{eq:polys}.  

It follows from the mean value property for harmonic functions that 
$$ U_{\mu_1}(z) =  \begin{cases} 
      \log |z|, & \text{ if } |z| > 1, \\
      0, & \text{ if } |z| \leq 1, \\
 \end{cases}$$
and
$$ U_{\mu_2}(z) = \begin{cases} 
      \log |z|, & \text{ if } |z| > r, \\
      \log r, & \text{ if } |z| \leq r. \\
 \end{cases}$$
Thus, we have that
$$U(z) := \max\{U_{\mu_1}(z), c U_{\mu_2}(z) \} =  
\begin{cases} 
      \log|z|, & \text{ if } |z| > r^c, \\
      \log r^c , & \text{ if } |z| \leq r^c. \\
\end{cases}$$
Therefore, it follows again from the mean value property for harmonic functions that the resultant measure $\rho$ given in Theorem \ref{thm:main} is the uniform probability measure on the circle of radius $r^c$ centered at the origin.  
\end{example}

\subsubsection{Heavy-tailed case}\label{sss:heavy}
We now consider the case when one of the measures $\mu_1, \ldots, \mu_m$ is not in $\mathcal{P}_{\lp}(\mathbb{C})$.  For simplicity, we focus on the case of only two polynomials, both having degree $n$.  To this end, let $\mu \in \mathcal{P}(\mathbb{C}) \setminus \mathcal{P}_{\lp}(\mathbb{C})$ and $\nu \in \mathcal{P}(\mathbb{C})$.  For each $n \geq 1$, define the random polynomials
\begin{equation} \label{eq:pnqn}
	p_n(z) := \prod_{j=1}^n (z - X_j), \qquad q_n(z) := \prod_{j=1}^n (z - Y_j), 
\end{equation}
where $X_1, Y_1, X_2, Y_2, \ldots$ is a sequence of independent random variables so that $X_j$ has distribution $\mu$ and $Y_j$ has distribution $\nu$ for each $j \geq 1$.  We are again interested in the the roots $z_1^{(n)}, \ldots, z_n^{(n)}$ of the sum $p_n + q_n$.  
In this case, the limiting behavior of the roots is determined by whichever measure, $\mu$ or $\nu$, has heavier tails.  
To this end, recall that the Mahler measure $M(f)$ of a monic polynomials $f(z) = \prod_{j=1}^n (z - \alpha_j)$ is given by 
\begin{equation} \label{eq:def:Mf}
	M(f) := \prod_{j=1}^n \max\{ |\alpha_j|, 1\}, 
\end{equation} 
and define 
\begin{equation} \label{eq:Sn}
	S_n := \frac{1}{n} \left( \log M(p_n) - \log M(q_n) \right) 
\end{equation} 
to be the normalized difference of the logarithmic Mahler measures of $p_n$ and $q_n$.  Notice that $\log M(p_n)$ is the sum of iid random variables.  However, the law of large numbers is not applicable here since $\mu \not\in \mathcal{P}_{\lp}(\mathbb{C})$.  

It follows from the results in \cite{MR336806, MR266315} that $S_n$ satisfies exactly one of the following three cases:
\begin{enumerate}[(i)]
\item \label{item:infty} with probability $1$, $\lim S_n = +\infty$;
\item \label{item:-infty} with probability $1$, $\lim S_n = -\infty$; or
\item \label{item:both} with probability $1$, $\limsup S_n = +\infty$ and $\liminf S_n = -\infty$.
\end{enumerate}
We refer to these three possibilities as cases \eqref{item:infty}, \eqref{item:-infty}, and \eqref{item:both}.  
Our main result in this setting is the following.

\begin{theorem}[Main result: heavy-tailed case] \label{thm:nolog}
Let $\mu \in \mathcal{P}(\mathbb{C}) \setminus \mathcal{P}_{\lp}(\mathbb{C})$ and $\nu \in \mathcal{P}(\mathbb{C})$.  Let $X_1, Y_1, X_2, Y_2, \ldots$ be a sequence of independent random variables so that $X_j$ has distribution $\mu$ and $Y_j$ has distribution $\nu$ for each $j \geq 1$.  For each $n \geq 1$, define the polynomials $p_n$ and $q_n$ as in \eqref{eq:pnqn}, and let 
\[ \rho_n := \frac{1}{n} \sum_{j=1}^n \delta_{z^{(n)}_j} \]
be the empirical measure constructed from the roots $z_1^{(n)}, \ldots, z_n^{(n)}$ of $p_n + q_n$.  The following convergence results hold depending on whether 
\[ S_n := \frac{1}{n} \left( \log M(p_n) - \log M(q_n) \right) \]
satisfies case \eqref{item:infty}, \eqref{item:-infty}, or \eqref{item:both}:
\begin{enumerate}
\item If case \eqref{item:infty} holds, then $\rho_n$ converges weakly almost surely to $\mu$ as $n \to \infty$.
\item If case \eqref{item:-infty} holds, then $\rho_n$ converges weakly almost surely to $\nu$ as $n \to \infty$.
\item If case \eqref{item:both} holds, then, with probability $1$, there is a subsequence along which $\rho_n$ converges weakly to $\mu$ and a separate subsequence along which $\rho_n$ converges weakly to $\nu$.  
\end{enumerate}
\end{theorem}

In cases \eqref{item:infty} and \eqref{item:-infty}, one of the Mahler measures dominates the other, and Theorem \ref{thm:nolog} shows that the zeros of the sum behave like those of the dominating polynomial; this behavior is not seen in Theorem \ref{thm:main} when the measures $\mu_1, \ldots, \mu_m$ have lighter tails.  Interestingly, in case \eqref{item:both} when $\mu \neq \nu$, Theorem \ref{thm:nolog} shows that, with probability $1$, $\rho_n$ does not converge weakly.  A numerical simulation of Theorem \ref{thm:nolog} is shown in Figure \ref{fig:heavy_tail}.  

Theorem \ref{thm:nolog} is only stated for the case of two polynomials ($m = 2$) where both polynomials have the same degree.  {When the polynomials no longer have the same degree or $m \geq 3$ one can no longer guarantee the three simple cases of (i), (ii) and (iii).  One can still apply Proposition \ref{pr:deterministic}---or a natural generalization to the $m \geq 3$ case---to handle specific cases when $m \geq 3$ or the polynomials are of different degrees, but we make no effort to classify all cases as we do in Theorem \ref{thm:nolog}.  As an example, if $m \geq 3$ and one Mahler measure dominates all others, then one can adapt Proposition \ref{pr:deterministic} to show that the corresponding measure of the dominant  Mahler measure determines the limit.}


\begin{figure}
    \centering
    \includegraphics[width=0.85\textwidth]{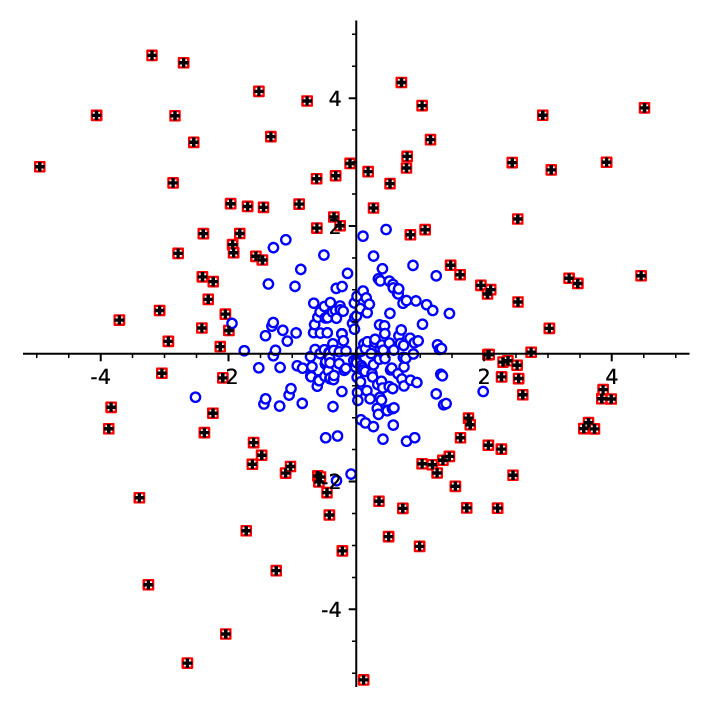}
    \caption{A numerical simulation of Theorem \ref{thm:nolog}.  The red squares represent the roots of $p_n$, which are chosen with respect to a rotationally symmetric distribution that does not have a logarithmic moment.  The blue circles are the roots of $q_n$, which are chosen according to the standard complex Gaussian distribution.  Both polynomials have degree $200$. The black crosses represent the roots of the sum $p_n + q_n$.  The image was cropped and does not display the largest roots (in magnitude) of $p_n$ or $p_n + q_n$. }
    \label{fig:heavy_tail}
\end{figure}

\subsection{Outline of the paper}
The rest of the paper is devoted to the proofs of Theorems \ref{thm:main} and \ref{thm:nolog}.  An overview of the proofs---along with a description of all the notation used in the paper---is presented in Section \ref{sec:overview}.  The main technical lemmas which establish Theorem \ref{thm:main} are presented in Sections \ref{sec:point} and \ref{sec:tightness}.  Theorem \ref{thm:nolog} is proved in Section \ref{sec:nolog}.  The appendix contains some auxiliary results.  

\subsection*{Acknowledgments}
The authors thank the anonymous referees for comments and corrections.  The third author thanks Magdalena Czubak and Andrew Campbell for useful conversations.

\section{Notation and an overview of the proofs} \label{sec:overview}

In this section, we give an overview of the proofs of Theorems \ref{thm:main} and \ref{thm:nolog}.  We begin with a description of the notation used throughout the paper.  

\subsection{Notation}
Unless otherwise noted, asymptotic notation (such as $O, o, \ll$) is used under the assumption that $n \to \infty$.  For example, $X=O(Y)$ and $X \ll Y$ denote the bound $|X| \leq CY$ for some constant $C > 0$ and all $n > C$, where $C$ is independent of $n$. Subscripts such as $X=O_k(Y)$ indicate that the constant $C$ depends on some parameter $k$, but $C$ may depend on the measures $\mu_1, \ldots, \mu_m$ (alternatively, $\mu$ and $\nu$) without denoting this dependence.  We use $X = o(Y)$ if $|X| \leq a_n Y$ for some sequence $\{a_n\}$ such that $\lim_{n \to \infty} a_n = 0$.  

We define the open disk of radius $r > 0$ centered at $z \in \mathbb{C}$ to be 
\[ B(z,r) := \{ w \in \mathbb{C} : |z - w| < r \}, \]
and set 
$B(r) := B(0, r)$.    We let $\partial B$ be the boundary of $B \subset \mathbb{C}$.  We will use $i$ to denote the imaginary unit. We also sometimes use $i$ as index; the reader will be able to tell the difference between these two uses of $i$ by context.  

 For a finite set $S$, we let $|S|$ denote its cardinality.  For $n \in \mathbb{N}$, $[n] = \{1, \ldots, n\}$ is the discrete interval.  
 
We let $\log(\cdot)$ be the natural logarithm. Let
\[ \lm x = \begin{cases} 
      |\log x|, & 0 \leq x \leq 1, \\
      0, & x \geq 1, 
   \end{cases}
   \quad\text {and }\quad
   \lp x = \begin{cases} 
      0, & 0 \leq x \leq 1, \\
      \log x, & x \geq 1,
   \end{cases} \]
denote the negative and positive parts of the logarithm.  Note that $\lm 0 = + \infty$.  

For a monic polynomial $f$, $M(f)$ is its Mahler measure, defined in \eqref{eq:def:Mf}.  
   
Let $\mathcal{P}(\mathbb{C})$ be the set of probability measures on $\mathbb{C}$.  We let $\mathcal{P}_{\lp}(\mathbb{C})$ denote the set of $\mu \in \mathcal{P}(\mathbb{C})$ such that
\[ \int_{\mathbb{C}} \lp |w| \d \mu(w) < \infty. \] 
Let $C_{c}^\infty(\mathbb{C})$ denote the set of all smooth functions $\varphi: \mathbb{C} \to \mathbb{C}$ with compact support; $\supp(\varphi)$ will denote the support of $\varphi$.  The Lebesgue measure on $\mathbb{C}$ is denoted by $\lambda$.  Unless otherwise noted, ``almost all'' and ``almost everywhere'' are with respect to the Lebesgue measure $\lambda$.  When a sequence of probability measures or random variables is tight, we will sometimes say the sequence is ``bounded in probability.''

\subsection{Overview of the proof of Theorem \ref{thm:nolog}}

The proof of Theorem \ref{thm:nolog} is based on Rouch\'e's theorem (see Exercise 24 in Chapter 10 of \cite{MR924157} for the general form of Rouch\'e's theorem used in the proof).  Suppose case \eqref{item:infty} holds. Recall that this means, with probability $1$, $\lim_{n \to \infty} S_n = + \infty$, where $S_n$ is defined in \eqref{eq:Sn}.  This means the Mahler measure of $p_n$ dominates the Mahler measure of $q_n$.  Since the Mahler measure is formed from the roots, we use this domination to show that
\begin{equation} \label{eq:domination}
	|p_n(z)| > |q_n(z)| 
\end{equation} 
for a sufficiently rich class of points $z \in \mathbb{C}$ and all $n$ sufficiently large.  In fact, for a class of Borel sets $B \subset \mathbb{C}$, we show that \eqref{eq:domination} holds for all $z \in \partial B$.  Hence, by Rouch\'e's theorem, $p_n$ and $p_n + q_n$ have the same number of zeros in $B$.  Since the empirical measure constructed from the roots of $p_n$ converges weakly almost surely to $\mu$ as $n \to \infty$, it will follow that $\rho_n$ also converges weakly to $\mu$ almost surely.  Case \eqref{item:-infty} follows a similar argument.  

When case \eqref{item:both} holds, with probability $1$, there is a subsequence under which $S_n$ converges to $+\infty$ and a separate subsequence under which $S_n$ converges to $-\infty$.  We apply a similar Rouch\'e's theorem argument as above to each of these subsequences.  The proof of Theorem \ref{thm:nolog} is presented in Section \ref{sec:nolog}.

\subsection{Proof of Theorem \ref{thm:main}}
We now outline the proof of Theorem \ref{thm:main} and its main technical lemmas.  In fact, using these lemmas, we will complete the proof of Theorem \ref{thm:main}; the technical lemmas are proved later in Sections \ref{sec:lpbound}, \ref{sec:point}, and \ref{sec:tightness}.  

For any smooth and compactly supported function $\varphi:\mathbb{C} \to \mathbb{C}$, it follows from Section 2.4.1 of \cite{HKYV} that
\begin{equation} \label{eq:HKYV}
\sum_{i=1}^n \varphi(z_i^{(n)}) = \frac{1}{2\pi} \int_{\mathbb{C}} \Delta \varphi(z) \log \left| \sum_{k=1}^m p_{n,k}(z) \right| \d \lambda(z). 
\end{equation}
Our goal is to show that
\begin{equation} \label{eq:convtoshow}
\frac{1}{ n} \int_{\mathbb{C}} \Delta \varphi(z) \log \left| \sum_{k=1}^m p_{n,k}(z) \right| \d \lambda(z) \longrightarrow \int_\mathbb{C} \Delta \varphi(z)\left( \max_{1 \leq k \leq m} {c_k}U_{\mu_k}(z) \right)\ d\lambda(z) 
\end{equation} 
in probability as $n \to \infty$.  This convergence will be established via the following dominated convergence theorem due to Tao and Vu \cite{TVcirc}.  

\begin{lemma}[Dominated convergence; Lemma 3.1 from \cite{TVcirc}] \label{lemma:dominated}
Let $(X, \rho)$ be a finite measure space.  For integers $n \geq 1$, let $f_n: X \to \mathbb{R}$ be random functions which are jointly measurable with respect to $X$ and the underlying probability space.  Assume that:
\begin{enumerate}
\item (uniform integrability) there exists $\delta > 0$ such that $\int_X |f_n(x)|^{1 + \delta} \d \rho(x)$ is bounded in probability (resp., almost surely);
\item \label{item:pointwise}(pointwise convergence) for $\rho$-almost every $x \in X$, $f_n(x)$ converges in probability (resp., almost surely) to zero.
\end{enumerate}
Then $\int_X f_n(x) \d \rho(x)$ converges in probability (resp., almost surely) to zero.  
\end{lemma}

In view of Lemma \ref{lemma:dominated}, the two main steps of the proof of Theorem \ref{thm:main} are contained in the following lemmata.  

\begin{lemma}\label{lem:point}
There is a measurable set $F \subset \C$ with $\lambda(F) = 0$ so that, for all $z \in \C \setminus F$, $\max_{1 \leq k \leq m} {c_k}U_{\mu_k}(z)$ is finite and 
\[ \frac{1}{n} \log \left| \sum_{k=1}^m p_{n,k}(z) \right| \longrightarrow \max_{1 \leq k \leq m} {c_k}U_{\mu_k}(z) \]
in probability as $n \to \infty$.
\end{lemma}

\begin{lemma}\label{lem:tightness}
For each $r > 0$, the sequence 
\[ \left\{ \frac{1}{n^2} \int_{B(r)} \left(\log \left| \sum_{k=1}^m p_{n,k}(z) \right|\right)^2 \d\lambda(z) \right\}_{n \geq 1} \] 
is bounded in probability.  
\end{lemma}

\begin{lemma} \label{lem:lpbound}
Let $U(z) := \max_{1 \leq k \leq m} {c_k}U_{\mu_k}(z)$.  Then for each $r > 0$, we have 
\[ \int_{B(r)} U^2(z) \d\lambda(z) < \infty. \]
\end{lemma}

With these results in hand, we can now complete the proof of Theorem \ref{thm:main}.
\begin{proof}[Proof of Theorem \ref{thm:main}]
Let $\varphi: \mathbb{C} \to \mathbb{C}$ be a smooth and compactly supported function.  Let $r > 0$ be sufficiently large so that the support of $\varphi$ is contained in $B(r)$.  Define
\[ f_n(z) := \Delta \varphi(z) \left( \frac{1}{n} \log \left| \sum_{k=1}^m p_{n,k}(z) \right| - \max_{1 \leq k \leq m} {c_k}U_{\mu_k}(z) \right). \]

Our goal is to apply Lemma \ref{lemma:dominated} to $f_n$ on the finite measure space consisting of the ball $B(r)$ with the Lebesgue measure.  We start by verifying the conditions of Lemma \ref{lemma:dominated}.  The pointwise convergence assumption follows from Lemma \ref{lem:point}.  In order to establish uniform integrability, we bound
\begin{align*}
	\int_{B(r)} |f_n(z)|^2 \d \lambda
	&\ll \| \Delta \varphi \|^2_{\infty} \left[ \int_{B(r)} \left( \frac{1}{n} \log \left| \sum_{k=1}^m p_{n,k}(z) \right|\right)^2 \d\lambda(z)  + \int_{B(r)} U^2(z) \d \lambda(z) \right], 
\end{align*}
where $\| \Delta \varphi\|_{\infty}$ is the $L^\infty$-norm of $\Delta \varphi$ and $U(z) := \max_{1 \leq k \leq m} {c_k}U_{\mu_k}(z)$.  This bound, together with Lemmas \ref{lem:tightness} and \ref{lem:lpbound}, establishes the uniform integrability assumption (with $\delta = 1$) in Lemma \ref{lemma:dominated}.  Applying Lemma \ref{lemma:dominated}, we conclude that
\[ \int_{\mathbb{C}} f_n(z) \d \lambda(z) = \int_{B(r)} f_n(z) \d \lambda(z) \longrightarrow 0 \]
in probability as $n \to \infty$, which establishes \eqref{eq:convtoshow}.  

Recall that 
\[ \rho_n := \frac{1}{n} \sum_{i=1}^n \delta_{z_i^{(n)}} \]
is the empirical measure constructed from the roots $z_1^{(n)}, \ldots, z_n^{(n)}$ of $\sum_{k=1}^m p_{n,k}$.  In view of \eqref{eq:HKYV}, we have shown that, for any $\varphi \in C_c^\infty(\mathbb{C})$, 
\begin{equation} \label{eq:convrhon}
	\frac{1}{n} \sum_{i=1}^n \varphi(z_i^{(n)}) = \int_{\mathbb{C}} \varphi \d \rho_n \longrightarrow \frac{1}{2 \pi} \int_\mathbb{C} \Delta \varphi(z)\left( \max_{1 \leq k \leq m} {c_k}U_{\mu_k}(z) \right)\d\lambda(z) 
\end{equation} 
in probability as $n \to \infty$.  

We now show the existence of a (deterministic) probability measure $\rho$ on $\mathbb{C}$ such that \eqref{eq:main} holds.  The proof given below for the existence of $\rho$ is based on standard results for subharmonic functions, and we refer the reader to the classic texts \cite{MR0460672, MR1049148} on the topic.   

Recall that $U_{\mu_1}, \ldots, U_{\mu_m}$ are subharmonic functions on $\mathbb{C}$.  Since $\log |\cdot|$ is locally Lebesgue integrable on $\mathbb{C}$, it follows from Fubini's theorem that $U_{\mu_1}, \ldots, U_{\mu_m}$ are finite almost everywhere (see also Lemma \ref{lemma:F_mu} below).  Without loss of generality, we will assume $U_{\mu_1}(0)$ is finite.  This is not a restriction since if $U_{\mu_1}(0) = -\infty$, we can find some $a \in \mathbb{C}$ where $U_{\mu_1}(a)$ is finite and repeat the proof above with the random variables $X_{i,k}$ replaced by $X_{i,k} - a$.  Indeed, shifting the roots of the polynomials $p_{n,k}$, $1 \leq k \leq m$ by $a$ shifts the logarithmic potentials $U_{\mu_1}, \ldots, U_{\mu_m}$ by $a$ and shifts each of the roots $z_1^{(n)}, \ldots, z_n^{(n)}$ of the sum $\sum_{k=1}^m p_{n,k}$ by $a$ as well.  

By Theorems 2.2.3 in \cite{Ransford}, $U(z) := \max_{1 \leq k \leq m} {c_k}U_{\mu_k}(z)$ is a subharmonic function on $\mathbb{C}$. 
Thus, from the results in Section 3.5.4 of \cite{MR0460672} (see also Section 3.7 of \cite{Ransford}), it follows that there exists a unique (deterministic) Radon measure $\rho$ on $\mathbb{C}$ (often referred to as the Riesz measure associated to $U$) such that \eqref{eq:main} holds.  Combining \eqref{eq:main} with \eqref{eq:convrhon}, we conclude that there exists a unique (deterministic) Radon measure $\rho$ on $\mathbb{C}$ so that, for all $\varphi \in C_c^\infty(\mathbb{C})$, 
\begin{equation} \label{eq:convsubprob}
	\int_{\mathbb{C}} \varphi \d\rho_n \longrightarrow \int_{\mathbb{C}} \varphi \d \rho 
\end{equation} 
in probability as $n \to \infty$. Since $\rho_n$ is a (random) probability measure for all $n$, it immediately follows that $\rho$ is a sub-probability measure (i.e., $\rho(\mathbb{C}) \leq 1$).  In order to complete the proof of the theorem, it suffices to show that $\rho$ is a probability measure.  Indeed, if $\rho$ is a probability measure, the convergence in \eqref{eq:convsubprob} can be extended to all bounded and continuous $\varphi: \mathbb{C} \to \mathbb{C}$ using standard truncation and approximation arguments, which would complete the proof.  

It remains to verify that $\rho$ is a probability measure.  To this end, for any subharmonic function $f: \mathbb{C} \to [-\infty, \infty)$ and any $r > 0$, define
\[ m(f, r) := \frac{1}{2 \pi} \int_0^{2 \pi} f(r e^{i \theta}) \d \theta. \]
Recall that every upper semicontinuous function on a compact set attains a maximum.  Since $U(z) := \max_{1 \leq k \leq m } c_k U_{\mu_k}(z)$ is subharmonic, it follows that
\begin{equation} \label{eq:bndU0}
	U_{\mu_1}(0) \leq U(0) \leq m(U, r) \leq \max_{|z| = r} U(z) < \infty 
\end{equation} 
for any $r > 0$, and hence $m(U, r)$ is finite for every $r > 0$ (where we used the assumption that $U_{\mu_1}(0)$ is finite).  Similarly, 
\begin{equation} \label{eq:bndUmu1}
	U_{\mu_1}(0) \leq m(U_{\mu_1}, r) \leq \max_{|z| = r} U_{\mu_1}(z) < \infty, 
\end{equation} 
and so $m(U_{\mu_1}, r)$ is finite for all $r > 0$.  In addition, since $U_{\mu_1}(z) \leq U(z)$ for all $z \in \mathbb{C}$, it follows that
\begin{equation} \label{eq:mUUbnd}
	m(U_{\mu_1}, r) \leq m(U, r) \quad \text{ for all } r > 0.  
\end{equation}

For $R > 1$, define
\[ \psi_R(z) := \begin{cases}
			\log R, & |z| \leq 1 \\
			\log \left( \frac{R}{|z|} \right), &1 < |z| < R \\
			0, & |z| \geq R.
		      \end{cases} \]
It follows from Eq. (3.5.7) in \cite{MR0460672} (see also Lemma 2.12 in \cite{MR2339369}) that
\begin{equation} \label{eq:rhoC}
	\int_{\mathbb{C}} \psi_R(z) \d \rho(z) = m( U, R) - m(U, 1) 
\end{equation} 
and 
\begin{equation} \label{eq:mu_1C}
	\int_{\mathbb{C}} \psi_R(z) \d \mu_1(z) = m(U_{\mu_1}, R) - m(U_{\mu_1}, 1). 
\end{equation} 
Since $\frac{1}{\log R} \psi_R \nearrow 1$ as $R \to \infty$, the monotone convergence theorem together with \eqref{eq:rhoC} and \eqref{eq:mu_1C} imply that
\begin{equation} \label{eq:rholim}
	\rho(\mathbb{C}) = \lim_{R \to \infty} \frac{ m(U, R) - m(U, 1) }{ \log R} = \lim_{R \to \infty} \frac{ m(U, R) }{ \log R} 
\end{equation}
and 
\begin{equation} \label{eq:mu_1lim}
	1 = \mu_1(\mathbb{C}) = \lim_{R \to \infty} \frac{ m(U_{\mu_1}, R) - m(U_{\mu_1}, 1) }{\log R} = \lim_{R \to \infty} \frac{m(U_{\mu_1}, R)}{\log R} 
\end{equation} 
since $m(U, 1)$ and $m(U_{\mu_1}, 1)$ are both finite and independent of $R$ (see \eqref{eq:bndU0} and \eqref{eq:bndUmu1} with $r = 1$).  Therefore, we conclude from \eqref{eq:mUUbnd}, \eqref{eq:rholim}, and \eqref{eq:mu_1lim} that
\[ \rho(\mathbb{C}) = \lim_{R \to \infty} \frac{ m(U, R)}{ \log R} \geq \lim_{R \to \infty} \frac{ m( U_{\mu_1}, R) }{ \log R} = 1. \]
Since we already showed $\rho$ is a sub-probability measure, we conclude that $\rho(\mathbb{C}) = 1$, and the proof is complete.  
\end{proof}

We prove Lemma \ref{lem:point} in Section \ref{sec:point}.  The proof of Lemma \ref{lem:tightness} is given in Section \ref{sec:tightness}.  We conclude this section with the proof of Lemma \ref{lem:lpbound}.

\subsection{Proof of Lemma \ref{lem:lpbound}} \label{sec:lpbound}
We will need the following result for the proof of Lemma \ref{lem:lpbound}.
\begin{proposition} \label{prop:logplusbound}
Let $\mu \in \mathcal{P}_{\lp}(\mathbb{C})$.  Then for any $r > 0$, there exists a constant $C_r > 0$ (depending on $\mu$ and $r$) so that
\[ \sup_{|z| < r} \int_{\mathbb{C}} \lp |z - w| \d\mu(w) \leq C_r. \]
\end{proposition}
\begin{proof}
Note that for $|z| < r$, we have $\lp |z-w| \leq \lp(r + |w|)$.  Thus, we obtain
\begin{align*}
	\int_{\mathbb{C}} \lp |z - w| \d\mu(w) &\leq \int_{\mathbb{C}} \lp (r + |w|) \d \mu(w) \\
	&\leq \lp(r + 1) + \int_{|w| \geq 1} \lp \left( |w| \left( \frac{r}{|w|} + 1 \right) \right) \d \mu(w) \\
	&\leq \lp(r + 1) + \int_{\mathbb{C}} \lp |w| \d\mu(w) + \int_{|w| \geq 1} \lp \left( \frac{r}{|w|} + 1 \right) \d\mu(w) \\
	&\leq 2 \lp(r+1) + \int_{\mathbb{C}} \lp |w| \d\mu(w).
\end{align*}
The claim now follows from the assumption that $\mu \in \mathcal{P}_{\lp}(\mathbb{C})$.  
\end{proof}

\begin{proof}[Proof of Lemma \ref{lem:lpbound}]
Since 
\[ U^2(z) \ll \sum_{k=1}^m U^2_{\mu_k}(z), \]
it suffices to show the bound for $U_{\mu_k}$.  
We bound 
\begin{align*}
	|U_{\mu_k}(z)| &\leq \int_{\mathbb{C}} \left|\log |z-w| \right| \d\mu_k(w) \\
	&= \int_{\mathbb{C}} \lp |z -w| \d\mu_k(w) + \int_{\mathbb{C}} \lp \frac{1}{|z-w|} \d \mu_{k} \\
	&=: \Xi^+_k(z) + \Xi^-_k(z).
\end{align*} 
It thus suffices to show local square integrability of each term on the right-hand side separately.  By Proposition \ref{prop:logplusbound}, we conclude that $\Xi^+_k$ is uniformly bounded on $B(r)$ and hence locally square integrable.  

For $\Xi^-_k$, we write
\begin{align*}
	\int_{B(r)} \left( \Xi^-_k(z) \right)^2 \d\lambda(z) = \int_{B(r)} \int_{\mathbb{C}} \int_{\mathbb{C}} \lp \frac{1}{|z-w_1|} \lp \frac{1}{|z-w_2|} \d \mu_k(w_1) \d \mu_k(w_2) \d \lambda(z).
\end{align*}
Since 
\[ \lp \frac{1}{|z-w_1|} \lp \frac{1}{|z-w_2|} \leq \lp^2 \frac{1}{|z-w_1|} + \lp^2 \frac{1}{|z-w_2|}, \]
by symmetry and Fubini's theorem we find
\begin{align*}
	\int_{B(r)} \left( \Xi^-_k(z) \right)^2 \d\lambda(z) &\leq 2 \int_{B(r)} \int_{\mathbb{C}} \lp^2 \frac{1}{|z-w|} \d \mu_k(w) \d \lambda(z) \\
	&= 2 \int_{\mathbb{C}} \int_{B(r)}  \lp^2 \frac{1}{|z-w|} \d \lambda(z) \d \mu_k(w) \\
	&\leq 2 \int_{\mathbb{C}} \int_{B(w,1)} \log^2 \frac{1}{|z-w|} \d \lambda(z) \d \mu_k(w). 
\end{align*}
By changing to polar coordinates, we obtain
\[ \int_{B(w,1)} \log^2 \frac{1}{|z-w|} \d \lambda(z) = 2 \pi \int_0^1 s \log^2 (1/s) \d s = \frac{\pi}{2} \]
for all $w \in \mathbb{C}$.  Therefore, we conclude that
\[ \int_{B(r)} \left( \Xi^-_k(z) \right)^2 \d\lambda(z) \leq \pi, \]
and the proof is complete.  
\end{proof}

\section{Proof of Lemma \ref{lem:point}} \label{sec:point}
This section is devoted to the proof of Lemma \ref{lem:point}.  We begin with some auxiliary results we will need for the proof.  This first result shows that the logarithmic potential is finite almost everywhere.  
\begin{lemma} \label{lemma:F_mu}
For any $\mu \in \mathcal{P}_{\lp}(\mathbb{C})$, there exists a Lebesgue measurable set $F_{\mu}$ so that $\lambda(F_\mu) = 0$ and for any $z \in \mathbb{C} \setminus F_{\mu}$
\[ \int_{\mathbb{C}} \left| \log |z -w | \right| \d\mu(w) < \infty. \]
\end{lemma}
\begin{proof}
Define 
\[ F_\mu = \left\{z \in \C: \int_{\mathbb{C}} \frac{1}{ |z -w| } \d\mu(w) = \infty  \right\}. \]
Observe that, for any $z \in \mathbb{C}$, 
\begin{align*} 
	\int_{\mathbb{C}} \frac{1}{ |z -w| } \d\mu(w) &= \int_{B(z,1)} \frac{1}{ |z -w| } \d\mu(w) + \int_{|w-z| \geq 1} \frac{1}{ |z -w| } \d\mu(w) \\
	&\leq \int_{B(z,1)} \frac{1}{ |z -w| } \d\mu(w) + 1.
\end{align*}
We now check that $\int_{B(z,1)} \frac{1}{ |z -w| } \d\mu(w)$ is finite for almost every $z \in \mathbb{C}$.  Indeed, by Fubini's theorem, we see that
\[ \int_{\C} \int_{B(z,1)} \frac{1}{ |z -w| } \d\mu(w)\d\lambda(z) = \int_{\C} \int_{B(w,1)} \frac{1}{ |z -w| } \d\lambda(z)\d\mu(w) = 2\pi, \]
and hence $\lambda(F_\mu) = 0$.

Now for $z \in \mathbb{C} \setminus F_{\mu}$, we write 
\begin{align*}
	\int_{\mathbb{C}} \left| \log |z -w | \right| \d\mu(w) &= \int_{\mathbb{C}}  \lp |z -w| \d\mu(w) + \int_{\mathbb{C}}  \lm |z -w| \d\mu(w).
\end{align*}
The first term on the right-hand side is finite by Proposition \ref{prop:logplusbound}.  The second term on the right-hand side is finite since $\lm |z -w| \leq \frac{1}{|z-w|}$ and $z \not\in F_{\mu}$.  
\end{proof}

\begin{remark}
As can be seen from the proof of Lemma \ref{lemma:F_mu}, the set $F_{\mu}$ contains the atoms of $\mu$.  However, $F_{\mu}$ may contain other points; for instance, when $\mu$ has rotationally invariant density 
\[ f(w) = \begin{cases} 
		\frac{1}{2 \pi |w|}, &\text{ if } 0 < |w| < 1, \\
		0, & \text{ otherwise}, 
	     \end{cases}	\]
it follows that $0 \in  F_{\mu}$.  
\end{remark}

We also require an anti-concentration bound for sums of iid random variables.  Similar anti-concentration inequalities have previously been used to study the zeros of random polynomials \cite{KZ,MR3283656,ORourke}.     

\begin{theorem}[Theorem 2.22 on p. 76 of \cite{Petrov}] \label{thm:1D}
Let $Y_1,Y_2,\ldots$ be i.i.d. copies of a non-degenerate real-valued random variable.  Then, for each fixed $t > 0$, we have 
\[ \lim_{n \to \infty} \sup_{u \in \mathbb{R}} \P\left(\left|\sum_{j = 1}^n Y_j - u\right| \leq t \right) = 0. \]
\end{theorem}

We will need the following result in order to apply Theorem \ref{thm:1D}.

\begin{lemma}\label{lem:non-constant}
Let $\mu$ be a non-degenerate probability measure on $\C$, and let $X$ be a random variable with distribution $\mu$.  Then there is a Lebesgue measurable set $E_\mu \subset \C$ with $\lambda(E_\mu) = 0$ so that, for any $z \in \C \setminus E_\mu$, $\log|z - X|$ is a non-degenerate random variable.
\end{lemma}
\begin{proof}
Suppose there exists $z_0 \in \mathbb{C}$ so that $\log |z_0 - X|$ is degenerate.  Thus, it must be the case that $|z_0 - x|$ is constant for all $x \in \supp(\mu)$.  By assumption, $\supp(\mu)$ contains at least $2$ points; if $\supp(\mu)$ is precisely two points, then set $E_\mu$ to be the line that is equidistant from the two points.  If $\supp(\mu)$ contains at least three points, then there is at most $1$ point equidistant from the three points, and we set $E_\mu$ to be this point.
\end{proof}

We now complete the proof of Lemma \ref{lem:point}.  

\begin{proof}[Proof of Lemma \ref{lem:point}]
For any subset $S$ of $[m]$, define the polynomial sum
\[ q_S(z) := \sum_{k \in S} p_{n,k}(z). \]
If $S$ is nonempty, the roots of $q_S$ can have at most countably many atoms.  Let $G$ be the collection of all atoms of the roots of $q_S$ as $S$ ranges over all nonempty subsets of $[m]$ and $n$ ranges over $\mathbb{N}$.  Then $G$ is at most countable and $\lambda(G) = 0$. 

For each $1 \leq k \leq m$, let $F_{\mu_k}$ and $E_{\mu_k}$ be the sets from Lemmas \ref{lemma:F_mu} and \ref{lem:non-constant}, respectively.  Set
\[ F := G \cup \bigcup_{k=1}^m \left( F_{\mu_k} \cup E_{\mu_k} \right). \]
It follows that $\lambda(F) = 0$.  

Fix $z \in \mathbb{C} \setminus F$.  Our goal is to show that
\begin{equation} \label{eq:point_show}
	\frac{1}{n} \log \left| \sum_{k=1}^m p_{n,k}(z) \right| \longrightarrow \max_{1 \leq k \leq m} c_k U_{\mu_k}(z) 
\end{equation} 
in probability as $n \to \infty$.  In order to do so, we will show that every subsequence of $\{n\}_{n \geq 1}$ has a further subsequence so that the convergence in \eqref{eq:point_show} holds along this further subsequence.  

Recall that $\{n_1\}_{n \geq 1}, \{n_2\}_{n \geq 1}, \ldots, \{n_m\}_{n \geq 1}$ are sequences indexed by $n$.  Consider an arbitrary subsequence of $\{n\}_{n \geq 1}$, which, for simplicity, we will denote as $\{n\}_{n \geq 1}$.  We now consider a further subsequence, again denoted by $\{n\}_{n \geq 1}$, so that, along this further subsequence, each $\{n_k\}_{n \geq 1}$ is either bounded or $\lim n_k = \infty$.  Let 
\[ S := \left\{ 1 \leq k \leq m : \lim n_k = \infty \right\}, \]
and let $S^c$ be the complement of $S$ in $[m]$.  Stated another way, $S^c$ contains the indices which correspond to bounded degree sequences.  We can then decompose our sum as
\begin{equation} \label{eq:sumdecomp}
	\sum_{k=1}^m p_{n,k}(z) = \sum_{k \in S} p_{n,k}(z) + q_{S^c}(z), 
\end{equation} 
where $q_{S^c}$ contains the ``low-degree'' polynomials; in particular, the degree of $q_{S^c}$ is bounded.  

The proof of the lemma is divided into four main steps. 
\begin{enumerate}
\item[\textbf{Step 1}]We first claim that for each $1 > \eps > 0$ and any $k,l \in S$ with $k \neq l$, 
\begin{equation}\label{eq:epsbound} \P\left( \eps < \frac{|p_{n,k}(z)|}{ |p_{n,l}(z)|} < \eps^{-1} \right) = o(1). 
\end{equation}

To establish \eqref{eq:epsbound}, we first note that, since $z \in \mathbb{C}\setminus F$, $z$ avoids the collection of all atoms of the roots of $p_{n,k}$ and $p_{n,l}$.  Thus, the ratio $p_{n,k}(z)/p_{n,l}(z)$ on the left-hand side of \eqref{eq:epsbound} is well-defined and nonzero, except on events which hold with probability zero, which we safely ignore for the remainder of the proof.  By conditioning on $p_{n,l}$, we have
\begin{align*}
	\P\left( \eps < \frac{|p_{n,k}(z)|}{ |p_{n,l}(z)|} < \eps^{-1} \right) &= \mathbb{P}\left(\left| \log|p_{n,k}(z)|-\log|p_{n,l}(z)| \right| \leq \log \epsilon^{-1} \right) \\
	&\leq \sup_{u \in \mathbb{R}} \P \left(\left| \sum_{i=1}^{n_k} \log|z - X_{i,k}| - u \right| \leq \log \epsilon^{-1} \right).
\end{align*} 
By Lemma \ref{lem:non-constant}, since $z \in \mathbb{C} \setminus F$,  $\log|z - X_{1,k}|$ is a non-degenerate random variable. Therefore, by Theorem \ref{thm:1D} and since $n_k \to \infty$ as $n \to \infty$, we have that 
\[ \lim_{n \to \infty} \sup_{u \in \mathbb{R}} \P\left(\left|\sum_{i = 1}^{n_k} \log|z - X_{i,k}| - u\right| \leq \log\epsilon^{-1} \right) = 0, \]
which establishes \eqref{eq:epsbound}.  

\item[\textbf{Step 2}] We claim that if $S^c$ is non-empty, then for each $1 > \eps > 0$ and any $k \in S$
\begin{equation} \label{eq:qscbound} \P\left( \eps < \frac{|p_{n,k}(z)|}{ |q_{S^c}(z)|} < \eps^{-1} \right) = o(1). 
\end{equation}

The proof of \eqref{eq:qscbound} is identical to the proof of \eqref{eq:epsbound}, except now we condition on $q_{S^c}$; we omit the details.  

\item[\textbf{Step 3}]We now use the previous two steps to show that
\begin{equation}\label{eq:logsumtologmax} \frac{1}{n} \log \left| \sum_{k=1}^m p_{n,k}(z) \right|  - \frac{1}{n} \log \max \left( \left \{ |p_{n,k}(z)| : k \in S \right\} \cup \{ |q_{S^c}(z)| \} \right) \longrightarrow 0 
\end{equation}
in probability as $n \to \infty$ when $S^c$ is nonempty and
\begin{equation} \label{eq:logsumtologmax2}
	\frac{1}{n} \log \left| \sum_{k=1}^m p_{n,k}(z) \right|  - \frac{1}{n} \log \max \left \{ |p_{n,k}(z)| : k \in [m] \right\} \longrightarrow 0 
\end{equation} 
in probability as $n \to \infty$ when $S^c$ is empty.  

To prove \eqref{eq:logsumtologmax}, assume $S^c$ is non-empty, and define
$$\Omega := \bigcup_{\substack{k,l \in S\\ k \not= l}} \left\{\frac{1}{2m} \leq \left| \frac{p_{n,k}(z)}{p_{n,l}(z)}\right| \leq 2m\right\} \cup \bigcup_{k \in S} \left\{\frac{1}{2m} \leq \left| \frac{p_{n,k}(z)}{q_{S^c}(z)}\right| \leq 2m\right\}. $$
By \eqref{eq:epsbound} and \eqref{eq:qscbound}, it follows that $\P(\Omega) = o(1)$, and hence it suffices to work on the event $\Omega^c$. 
Fix a realization $\omega \in \Omega^c$.  We consider two cases.  If, for some $s \in S$, 
\[ |p_{n,s}(z)| =  \max \left( \left \{ |p_{n,k}(z)| : k \in S \right\} \cup \{ |q_{S^c}(z)| \} \right), \]
then, for the fixed realization $\omega$, 
\begin{equation} \label{eq:pnkratio}
	\left| \frac{p_{n,k}(z)}{p_{n,s}(z)}\right| < \frac{1}{2m}, \quad k \in S, k \not= s
\end{equation} 
and
\begin{equation} \label{eq:qScratio}
	\left| \frac{ q_{S^c}(z) }{ p_{n,s}(z) } \right| < \frac{1}{2m}. 
\end{equation} 
Thus, factoring out $p_{n,s}(z)$, we have that 
\begin{align*}
\frac{1}{n} \log \left| \sum_{k=1}^m p_{n,k}(z) \right| &= \frac{1}{n}\log\left| \sum_{k \in S} p_{n,k}(z) + q_{S^c}(z)\right| \\
&= \frac{1}{n}\log|p_{n,s}(z)| + \frac{1}{n}\log\left|1+\sum_{\substack{k \in S \\ k \neq s}} \frac{p_{n,k}(z)}{p_{n,s}(z)} + \frac{q_{S^c}(z)}{p_{n,s}(z)} \right| \\ 
&=  \frac{1}{n}\log|p_{n,s}(z)| + O\left(\frac{1}{n}\right),
\end{align*}
where we used \eqref{eq:sumdecomp} in the first step and \eqref{eq:pnkratio} and \eqref{eq:qScratio} in the last step.  In the second case, if 
\[ |q_{S^c}(z)| =  \max \left( \left \{ |p_{n,k}(z)| : k \in S \right\} \cup \{ |q_{S^c}(z)| \} \right), \]
then
\begin{equation} \label{eq:pnkratio2}
	\left| \frac{p_{n,k}(z)}{q_{S^c}(z)}\right| < \frac{1}{2m}, \quad k \in S.
\end{equation}
Thus, factoring out $q_{S^c}(z)$, we have
\begin{align*}
\frac{1}{n} \log \left| \sum_{k=1}^m p_{n,k}(z) \right| &= \frac{1}{n}\log\left| \sum_{k \in S} p_{n,k}(z) + q_{S^c}(z)\right| \\
&= \frac{1}{n}\log|q_{S^c}(z)| + \frac{1}{n}\log\left|1+\sum_{k \in S} \frac{p_{n,k}(z)}{q_{S^c}(z)} \right| \\ 
&=  \frac{1}{n}\log|q_{S^c}(z)| + O\left(\frac{1}{n}\right),
\end{align*}
where we used \eqref{eq:pnkratio2} in the last step.  

Therefore, for the fixed realization $\omega$, we conclude that
\[ \frac{1}{n} \log \left| \sum_{k=1}^m p_{n,k}(z) \right| = \frac{1}{n} \log \max \left( \left \{ |p_{n,k}(z)| : k \in S \right\} \cup \{ |q_{S^c}(z)| \} \right) + O \left( \frac{1}{n} \right), \]
and the convergence in \eqref{eq:logsumtologmax} holds on $\Omega^c$.  

Lastly, if $S^c$ is empty, an analogous argument establishes \eqref{eq:logsumtologmax2}; we omit the details.

\item[\textbf{Step 4}]Finally, we use the law of large numbers to show that 
\begin{equation}\label{eq:logmaxtomax} \frac{1}{n} \log \max \left( \left \{ |p_{n,k}(z)| : k \in S \right\} \cup \{ |q_{S^c}(z)| \} \right) \longrightarrow \max_{1 \leq k \leq m} c_k U_{\mu_k}(z)
\end{equation}
in probability as $n \to \infty$ when $S^c$ is nonempty and 
\begin{equation}\label{eq:logmaxtomax2} \frac{1}{n} \log \max \left \{ |p_{n,k}(z)| : k \in [m] \right\} \longrightarrow \max_{1 \leq k \leq m} c_k U_{\mu_k}(z)
\end{equation}
in probability as $n \to \infty$ when $S^c$ is empty.  

Indeed, for $k \in S$, we have that
\[ \frac{1}{n}\log|p_{n,k}(z)| = \frac{1}{n}\sum_{i=1}^{n_k} \log|z-X_{i,k}| = \frac{n_k}{n}\left(\frac{1}{n_k}\sum_{i=1}^{n_k} \log|z-X_{i,k}|\right), \]
and so by the law of large numbers (and the assumptions that $z \in \mathbb{C}\setminus F$ and $\mu_{k} \in \mathcal{P}_{\log_+}(\C)$)
$$\frac{n_k}{n}\left(\frac{1}{n_k}\sum_{i=1}^{n_k} \log|z-X_{i,k}|\right) \longrightarrow c_kU_{\mu_k}(z)$$
almost surely as $n \to \infty$. Thus, by continuity of the pointwise maximum, 
\begin{equation} \label{eq:maxkS}
	\frac{1}{n}\log \max \left \{ |p_{n,k}(z)| : k \in S \right\} \longrightarrow \max_{k \in S} c_kU_{\mu_k}(z)
\end{equation} 
almost surely as $n \to \infty$.  
If $S^c$ is empty, then $S = [m]$ and \eqref{eq:logmaxtomax2} follows from \eqref{eq:maxkS}.  

Assume $S^c$ is nonempty.  Recall that 
$$q_{S^c} = \sum_{k \in S^c} p_{n,k}(z)$$
and 
$c_k = \lim_{n \to \infty} \frac{n_k}{n} = 0$ for all $k \in S^c$. 
We have that
$$\frac{1}{n}\log|q_{S^c}(z)| = \frac{1}{n}\log \left| \alpha \prod_{i=1}^{n_{j_n}} (z-\xi_i) \right| = \frac{\log \alpha}{n} + \frac{1}{n} \sum_{i=1}^{n_{j_n}} \log|z-\xi_i|,$$
where $n_{j_n}$ is the degree of $q_{S^c}$, $\alpha$ is the leading coefficient of $q_{S^c}$, and $\xi_1, \ldots, \xi_{n_j}$ denote the roots of $q_{S^c}$.  
Since the roots of $q_{S^c}$ avoid $z$ with probability $1$ and $n_{j_n}$ is bounded, it follows that 
$$ \frac{1}{n} \sum_{i=1}^{n_{j_n}} \log|z-\xi_i| \longrightarrow 0 $$
almost surely as $n \to \infty$.  In addition, since $1 \leq \alpha \leq m-1$, we have that
\[ \frac{\log \alpha}{n} \longrightarrow 0 \]
as $n \to \infty$, and hence 
\begin{equation} \label{eq:qScconvfinal}
	\frac{1}{n}\log|q_{S^c}(z)| \longrightarrow 0 
\end{equation} 
almost surely.  Since $c_kU_{\mu_k}(z) = 0$ for any $k \in S^c$, \eqref{eq:logmaxtomax} follows from \eqref{eq:maxkS}, \eqref{eq:qScconvfinal}, and the continuity of the pointwise maximum.  
\end{enumerate}

Combining \eqref{eq:logsumtologmax} and \eqref{eq:logmaxtomax} (alternatively, \eqref{eq:logsumtologmax2} and \eqref{eq:logmaxtomax2} when $S^c$ is empty) completes the proof of the lemma. 
\end{proof}

\section{Proof of Lemma \ref{lem:tightness}} \label{sec:tightness}

We now turn to the proof of Lemma \ref{lem:tightness}.  The proofs in this section are similar to the arguments given by Kabluchko in \cite{MR3283656}.  Recall that $z_1^{(n)}, \ldots, z_n^{(n)}$ are the roots of $\sum_{k=1}^m p_{n,k}$.  Since the random variables $|z_1^{(n)}|, \ldots, |z_n^{(n)}|$ can have at most countably many atoms, it follows that, for all but countably many values of $s \in [0, \infty)$, there are no roots of $\sum_{k=1}^m p_{n,k}$ of modulus $s$ for all $n \geq 1$ with probability $1$.  Throughout this section, when we write $r$ and $R$, we assume these two values satisfy this property.  

The proof of Lemma \ref{lem:tightness} is based on the Poisson--Jensen formula for 
\[ \log \left| \sum_{k=1}^m p_{n,k}(z) \right|, \]
see, for example, Chapter 8 of \cite{MR0181738}.  Let $R > r > 0$, and let $y_{1}^{(n)}, \ldots, y_{\ell_n}^{(n)}$ be the zeros of $\sum_{k=1}^m p_{n,k}$ in the disk $B(R)$ (repeated according to multiplicity).  While $\ell_n$ is a random variable, it will always satisfy the deterministic bound $\ell_n \leq n$.  

For any $z \in B(r)$ which is not a zero of $\sum_{k=1}^m p_{n,k}$, the Poisson--Jensen formula states that 
\begin{equation}\label{eq:L-Jensen}
	\log \left| \sum_{k=1}^m p_{n,k}(z)\right| = I_n(z;R) + \sum_{\ell = 1}^{\ell_n} \log\left| \frac{R(z - y_{\ell}^{(n)})}{R^2 - \overline{y_{\ell}^{(n)}}z} \right|, 
\end{equation}
where 
\begin{equation}\label{eq:Poisson-int}
	I_n(z;R) := \frac{1}{2\pi}\int_0^{2\pi} \log \left| \sum_{k=1}^m p_{n,k}(R e^{i\theta}) \right| P_R(|z|,\theta - \arg z)\d\theta
\end{equation}
and $P_R$ denotes the Poisson kernel 
\begin{equation}\label{eq:Poisson-kernel}
	P_R(t,\theta) := \frac{R^2 - t^2}{R^2 + t^2 - 2R t \cos(\theta)} \qquad t \in [0,R], \theta \in [0,2\pi]. 
\end{equation}

We will need the following observation.  

\begin{lemma}\label{lem:I-circle}
For each $R > 0$, there is a constant $C > 0$ (depending only on $R, m, \mu_1, \ldots, \mu_m$) so that 
\[ \E\left[\frac{1}{n}\int_0^{2\pi} \lp \left| \sum_{k=1}^m p_{n,k} (R e^{i\theta}) \right| \d\theta \right] \leq C. \]
\end{lemma}
\begin{proof}
Since
\begin{align*}
	\lp \left| \sum_{k=1}^m p_{n,k} (R e^{i\theta}) \right| &\leq \lp \left( \sum_{k=1}^m \left| p_{n,k}(R e^{i\theta}) \right| \right) \\
	&\leq \lp \left( m \max_{1 \leq k \leq m} |p_{n,k}(R e^{i\theta})| \right) \\
	&\leq \sum_{k=1}^m \lp \left |p_{n,k}(R e^{i\theta}) \right| + \log m, 
\end{align*}
the result follows by applying Fubini's theorem and Proposition \ref{prop:logplusbound}.  
\end{proof}

We now obtain an upper bound for $I_n(z;R)$ that holds uniformly for $z \in B(r)$.  

\begin{lemma}\label{lem:I-UB}
For every $0 < r < R$ there is a constant $C > 0$ (depending only on $r, R, m, \mu_1, \ldots, \mu_m$) so that 
\[ \P \left( \frac{1}{n}\sup_{z \in B(r)} I_n(z;R) \geq t \right) \leq \frac{C}{t} \]
for all $t > 0$. 
\end{lemma}
\begin{proof}
By \eqref{eq:Poisson-kernel}, there exists $M > 1$ (depending on $R$ and $r$) so that 
\[ \frac{1}{M} \leq P_R(|z|,\theta) \leq M \qquad \text{for all } z\in B(r), \theta \in [0,2\pi]. \]
Thus, we have
\begin{align*}
	I_n(z;R) &\leq \frac{1}{2\pi}\int_0^{2\pi} \lp \left| \sum_{k=1}^m p_{n,k}(R e^{i\theta}) \right| P_R(|z|,\theta - \arg z)\d\theta \\
	&\leq \frac{M}{2 \pi} \int_0^{2\pi} \lp \left| \sum_{k=1}^m p_{n,k}(R e^{i\theta}) \right| \d\theta,
\end{align*}
and hence it suffices to bound
\[ \P \left( \frac{M}{2 \pi n} \int_0^{2\pi} \lp \left| \sum_{k=1}^m p_{n,k}(R e^{i\theta}) \right| \d\theta \geq t \right). \]
The claim now follows from Markov's inequality and Lemma \ref{lem:I-circle}.  
\end{proof}

We will use the Poisson--Jensen formula again to obtain a lower bound on $I_n(z;R)$.  We start with the case when $z = 0$.  Recall that $F$ is the exceptional set from Lemma \ref{lem:point} that has  Lebesgue measure zero.  We will assume throughout that $0 \not\in F$.  This is not a restriction since if $0 \in F$, we can find some $a \in \mathbb{C} \setminus F$ and prove Theorem \ref{thm:main} with the random variables $X_{i,k}$ replaced by $X_{i,k} - a$.  Indeed, shifting the roots of the polynomials $p_{n,k}$, $1 \leq k \leq m$ by $a$ shifts each of the roots $z_1^{(n)}, \ldots, z_n^{(n)}$ of the sum $\sum_{k=1}^m p_{n,k}$ by $a$ as well.  

\begin{lemma}\label{lem:I-0}
Assume $0 \notin F$ and $R > 0$.  There is a constant $A > 0$ (depending only on $\mu_1, \ldots, \mu_m$ and $c_1, \ldots, c_m$) so that 
\[ \lim_{n \to \infty} \P\left(\frac{1}{n} I_n(0; R) \leq - A \right) = 0. \]
\end{lemma}
\begin{proof}
From the Poisson--Jensen formula \eqref{eq:L-Jensen} with $z = 0$, we have 
\[ \log \left| \sum_{k=1}^m p_{n,k}(0)\right| = I_n(0;R) + \sum_{\ell = 1}^{\ell_n} \log\left| \frac{ y_{\ell}^{(n)} }{R} \right|, \]
where $y_{1}^{(n)}, \ldots, y_{\ell_n}^{(n)}$ are the zeros of $\sum_{k=1}^m p_{n,k}$ in the disk $B(R)$.  Thus, it follows that
\[ \frac{1}{n}  I_n(0;R) \geq \frac{1}{n} \log \left| \sum_{k=1}^m p_{n,k}(0)\right|, \]
and so
\[ \P\left(\frac{1}{n} I_n(0; R) \leq - A \right) \leq \P \left( \frac{1}{n} \log \left| \sum_{k=1}^m p_{n,k}(0)\right| \leq -A \right) \]
for any $A > 0$.  

Since $0 \not\in F$, by Lemma \ref{lem:point}, we obtain
\[ \frac{1}{n} \log \left| \sum_{k=1}^m p_{n,k}(0) \right| \longrightarrow \max_{1 \leq k \leq m} {c_k}U_{\mu_k}(0) \]
in probability as $n \to \infty$.  
Therefore, choosing $A > 0$ so that 
\[ -A < \max_{1 \leq k \leq m} {c_k}U_{\mu_k}(0) \] 
completes the proof.  
\end{proof}

We now extend Lemma \ref{lem:I-0} to a lower bound on $\inf_{z \in B(r)}I(z;R)$.  

\begin{lemma}\label{lem:I-gen}
Assume $0 \not\in F$.  For any $0 < r < R$, there is a constant $B > 0$ (depending only on $r, R, m, \mu_1, \ldots, \mu_m$ and $c_1, \ldots, c_m$) so that for all $K \geq 1$ we have  
\[ \limsup_{n \to \infty} \P\left(\frac{1}{n} \inf_{z \in B(r)}I(z;R) \leq -B K \right) \leq 1/K. \]
\end{lemma}
\begin{proof}
By \eqref{eq:Poisson-kernel}, there exists $M > 1$ (depending on $R$ and $r$) so that 
\[ \frac{1}{M} \leq P_R(|z|,\theta) \leq M \qquad \text{for all } z\in B(r), \theta \in [0,2\pi]. \]
From \eqref{eq:Poisson-int}, we bound
\begin{align*}
	\frac{2\pi}{n} I_n(z;R) &\geq \frac{1}{n}\int_0^{2\pi} \left( \frac{1}{M} \lp \left| \sum_{k=1}^m p_{n,k}(R e^{i\theta}) \right| - M \lm \left| \sum_{k=1}^m p_{n,k}(R e^{i\theta}) \right|\right)\d\theta \\
	&= \frac{2\pi M}{n} I_n(0;R) - \left(M - \frac{1}{M} \right) \frac{1}{n} \int_0^{2\pi} \lp \left| \sum_{k=1}^m p_{n,k}(R e^{i\theta}) \right|\d\theta.  
\end{align*}
Therefore, we find
\begin{align*}
	\P&\left(\frac{1}{n} \inf_{z \in B(r)}I(z;R) \leq -B K \right) \\
	&\qquad\qquad\leq \P \left( \frac{2\pi M}{n} I_n(0;R) \leq \frac{-BK}{2} \right) \\
	&\qquad\qquad\qquad+ \P \left( \left(M - \frac{1}{M} \right) \frac{1}{n} \int_0^{2\pi} \lp \left| \sum_{k=1}^m p_{n,k}(R e^{i\theta}) \right|\d\theta \geq \frac{BK}{2} \right). 
\end{align*}
By Lemma \ref{lem:I-0}, the first term on the right-hand side converges to zero as $n$ tends to infinity for $B > 0$ sufficiently large and any $K \geq 1$.  Applying Markov's inequality and Lemma \ref{lem:I-circle} to the second term with $B$ sufficiently large completes the proof of the lemma.  
\end{proof}

\begin{lemma} \label{lem:I-tight}
Assume $0 \not\in F$.  For $0 < r < R$, the sequence 
\[ \left\{ \frac{1}{n^2} \int_{B(r)} \left[ I_n(z;R) \right]^2 \d\lambda(z) \right\}_{n \geq 1} \]
is tight.  
\end{lemma}
\begin{proof}
We want to show that for all $\eps > 0$ there exists $t > 0$ so that
\begin{equation}\label{eq:tightbound}
	\P\left(\frac{1}{n^2} \int_{B(r)} \left[ I_n(z;R) \right]^2 \d\lambda(z) \geq t\right) \leq \eps
\end{equation}
for all $n \geq 1$.
Clearly \eqref{eq:tightbound} holds when $\eps \geq 1$, so let $0 < \eps < 1$. 
For any $n \geq 1$, we can bound
\begin{equation}\label{eq:intibound}
\int_{B(r)} \left[ I_n(z;R) \right]^2 \d\lambda(z) \leq \pi r^2\max\left\{\left[\inf_{z \in B(r)}  I_n(z;R) \right]^2, \left[  \sup_{z \in B(r)} I_n(z;R) \right]^2\right\}.
\end{equation}
Observe that if the maximum on the right-hand side of  \eqref{eq:intibound} is given by 
\[ \left[\sup_{z \in B(r)} I_n(z;R)\right]^2, \]
we must have that  $\sup_{z \in B(r)} I_n(z;R) \geq 0$. Similarly, if the maximum is given by  $\left[\inf_{z \in B(r)}  I_n(z;R) \right]^2$, then $\inf_{z \in B(r)} I_n(z;R) \leq 0$. 
Let $C$ be the positive constant from Lemma \ref{lem:I-UB}. Let $t_1 > 0$ be chosen large enough such that $\frac{Cr\sqrt{\pi}}{\sqrt{t_1}} < \frac{\eps}{2}.$  
Then we can bound
\begin{align*}
\mathbb{P}&\left(\frac{\pi r^2}{n^2} \left[  \sup_{z \in B(r)} I_n(z;R) \right]^2 \geq t_1, \sup_{z \in B(r)} I_n(z;R) \geq 0\right) \\
&\qquad\qquad\qquad\qquad\qquad= \mathbb{P}\left(\frac{1}{n} \sup_{z \in B(r)} I_n(z;R) \geq \sqrt{\frac{t_1}{\pi r^2}}\right)
 \leq \frac{Cr\sqrt{\pi}}{\sqrt{t_1}}
 < \frac{\eps}{2},
\end{align*} 
where the first inequality follows from Lemma \ref{lem:I-UB}. 
Similarly, let $B$ be the positive constant given in Lemma \ref{lem:I-gen}. Let $t_2 > 0$ be chosen large enough such that
$\frac{Br\sqrt{\pi}}{\sqrt{t_2}} < \frac{\eps}{4}$. Note that since $0 < \eps < 1$, $\frac{\sqrt{t_2}}{Br\sqrt{\pi}} > 1$. We have
\begin{align*}
\mathbb{P}&\left(\frac{\pi r^2}{n^2} \left[  \inf_{z \in B(r)} I_n(z;R) \right]^2 \geq t_2, \inf_{z \in B(r)} I_n(z;R) \leq 0\right) \\
&\qquad\qquad\qquad\qquad\qquad\qquad= \mathbb{P}\left(\frac{1}{n} \inf_{z \in B(r)} I_n(z;R) \leq -\sqrt{\frac{t_2}{\pi r^2}}\right) \\
&\qquad\qquad\qquad\qquad\qquad\qquad= \mathbb{P}\left(\frac{1}{n} \inf_{z \in B(r)} I_n(z;R) \leq (-B)\left(\frac{\sqrt{t_2}}{Br\sqrt{\pi}}\right)\right). 
\end{align*}
By Lemma  \ref{lem:I-gen}, there exists $N \in \mathbb{N}$ so that for all $n \geq N$, we have
\[ \mathbb{P}\left(\frac{1}{n} \inf_{z \in B(r)} I_n(z;R) \leq (-B)\left(\frac{\sqrt{t_2}}{Br\sqrt{\pi}}\right)\right) \leq \frac{Br\sqrt{\pi}}{\sqrt{t_2}} + \frac{\eps}{4}< \frac{\eps}{2} \]
by the choice of $t_2$.  
For the remaining $1 \leq n \leq N$, by continuity of measure, there exists $t_3 > 0$ so that 
\[ \sup_{1 \leq n \leq N} \P\left(\frac{1}{n} \inf_{z \in B(r)}I_n(z;R) \leq -t_3 \right) < \frac{\eps}{2}. \]
Letting $t = \max\{t_1, t_2, t_3\}$, for any $n \geq 1$, we have that
\begin{align*}
\mathbb{P}&\left(\frac{1}{n^2} \int_{B(r)} \left[ I_n(z;R) \right]^2 \d\lambda(z) \geq t\right) \\
&\qquad\qquad\qquad\qquad\leq \mathbb{P}\left(\frac{\pi r^2}{n^2} \left[  \sup_{z \in B(r)} I_n(z;R) \right]^2 \geq t, \sup_{z \in B(r)} I_n(z;R) \geq 0\right) \\
&\qquad\qquad\qquad\qquad\qquad + \mathbb{P}\left(\frac{\pi r^2}{n^2} \left[  \inf_{z \in B(r)} I_n(z;R) \right]^2 \geq t, \inf_{z \in B(r)} I_n(z;R) \leq 0\right) \\
&\qquad\qquad\qquad\qquad< \frac{\eps}{2} + \frac{\eps}{2} \\
&\qquad\qquad\qquad\qquad= \eps,
\end{align*}
which establishes \eqref{eq:tightbound}. 
\end{proof}

We are now in position to complete the proof of Lemma \ref{lem:tightness}.  

\begin{proof}[Proof of Lemma \ref{lem:tightness}]
Recall that $F$ denotes the exceptional set from Lemma \ref{lem:point}.  As noted in the paragraph preceding Lemma \ref{lem:I-0}, without loss of generality, we may assume $0 \not\in F$.  We assume $0 \not\in F$ for the remainder of the proof.  

By \eqref{eq:L-Jensen} and the Cauchy--Schwarz inequality, 
\begin{align} \label{eq:inteqbnd}
	\frac{1}{n^2} \int_{B(r)} \log^2 \left| \sum_{k=1}^m p_{n,k}(z) \right| \d\lambda(z) &\ll \frac{1}{n^2} \int_{B(r)} \left[ I_n(z;R) \right]^2 \d\lambda(z) \\
	&\quad+ \frac{ \ell_n }{ n^2 } \sum_{\ell =1}^{\ell_n}  \int_{B(r)} \log^2 \left| \frac{R(z - y_{\ell}^{(n)})}{R^2 - \overline{y_{\ell}^{(n)}}z} \right| \d\lambda(z), \nonumber
\end{align}
where $y_{1}^{(n)}, \ldots, y_{\ell_n}^{(n)}$ are the zeros of $\sum_{k=1}^m p_{n,k}$ in the disk $B(R)$. 

Note that for any $y \in B(R)$ and $z \in B(r)$, $|R^2 - \bar{y}z|$ is uniformly bounded below.  Since $\log | \cdot |$ is locally square integrable on $\mathbb{C}$, this implies that 
\[ \sup_{y \in B(R)} \int_{B(r)} \log^2 \left|\frac{R(z - y)}{R^2 - \overline{y}z} \right|\d\lambda(z) = O_{R, r}(1). \]
Thus, bounding $\ell_n \leq n$ yields the deterministic bound
\[ \frac{\ell_n}{n^2} \sum_{\ell = 1}^{\ell_n} \int_{B(r)} \log^2 \left| \frac{R(z - y_{\ell}^{(n)})}{R^2 - \overline{y_{\ell}^{(n)}}z} \right| \d\lambda(z) \leq \sup_{y \in B(R)} \int_{B(r)} \log^2 \left|\frac{R(z - y)}{R^2 - \overline{y}z} \right|\d\lambda(z) = O_{R, r}(1). \] 

Lemma \ref{lem:I-tight} implies that the sequence 
\[ \left\{ \frac{1}{n^2} \int_{B(r)} \left[ I_n(z;R) \right]^2 \d\lambda(z) \right\}_{n \geq 1} \]
is tight.  Since a tight sequence plus a deterministically bounded sequence is tight, we conclude from \eqref{eq:inteqbnd} that
\[ \left\{ \frac{1}{n^2} \int_{B(r)} \log^2 \left| \sum_{k=1}^m p_{n,k}(z) \right| \d \lambda(z) \right\}_{n\geq 1} \]
is tight, and the proof is complete.
\end{proof}

\section{Proof of Theorem \ref{thm:nolog}} \label{sec:nolog}
This section is devoted to the proof of Theorem \ref{thm:nolog}.  We begin with some lemmata which we will need for the proof.  

\begin{lemma}\label{lem:Fmur}
If $\mu \in \mathcal{P}(\mathbb{C})$, then for each point $x \in \C$ and Lebesgue almost all $r \in (0, \infty)$, we have 
\[ \int_{\C} \lp \frac{1}{||w  - x| - r|} \d\mu(w) < \infty. \]
\end{lemma}
\begin{proof}
By shifting the measure $\mu$, we may assume without loss of generality that $x = 0$.  In this case, by Fubini's theorem, we have
\begin{align*}
	\int_{0}^\infty \int_{\C} \lp \frac{1}{||w| - r|}\d\mu(w)\d r &= \int_{\C}\int_{0}^\infty  \lp \frac{1}{||w| - r|}\d r \d\mu(w) \\
	&\leq \int_{\C} \int_{-1}^1 \log|1/s| \d s \d\mu(w) = 2, 
\end{align*}
where we used the fact $\int_{-1}^1 \log|1/s| \d s = 2$.  We conclude that
\[ \int_{\C} \lp \frac{1}{||w| - r|}\d\mu(w) < \infty \]
for Lebesgue almost every $r > 0$.  
\end{proof}

For $\mu \in \mathcal{P}(\mathbb{C})$ and each $x \in \C$, define 
\[ F_x^{\mu} := \left\{r > 0 : \int_{\C} \log_+ \left|\frac{1}{|w - x| - r} \right|\,d\mu(w)  = \infty\right\}. \]
By Lemma \ref{lem:Fmur}, the set $F_x^{\mu}$ has Lebesgue measure $0$ for every $x \in \mathbb{C}$.  In addition, if $X$ is a random variable with distribution $\mu$, then $F_x^{\mu}$ contains the atoms of $|X - x|$.  

\begin{lemma} \label{lem:replace-Mahler}
Let $\mu \in \mathcal{P}(\mathbb{C})$, and define the random polynomial
\[ f_n(z) := \prod_{j=1}^n (z - X_j), \]
where $X_1, X_2, \ldots$ are iid random variables with distribution $\mu$.  Then for each $x \in \mathbb{C}$ and all $r \in (0, \infty) \setminus F_x^{\mu}$, almost surely
\[ \limsup_{n\to\infty} \frac{1}{n} \max_{|z| = r} \left| \log |f_n(x + z)| - \log M(f_n) \right| < \infty, \]
where $M(f_n)$ is the Mahler measure of $f_n$, defined in \eqref{eq:def:Mf}.  
\end{lemma}
\begin{proof}
We first prove the lemma in the case when $x = 0$.  
Let $r \in (0, \infty) \setminus F_0^{\mu}$.  For $|z| = r$, we upper bound 
\begin{align*}
	\log|f_n(z)| &= \sum_{j=1}^n \log |z - X_j| \\
	&\leq \sum_{j=1}^n \log (1 + r + |X_j|) \\
	&\leq \sum_{j : |X_j| \leq 1} \log (2 + r) + \sum_{j : |X_j| > 1} \left( \log |X_j| + \log (2+r) \right) \\
	&\leq n\log(2 + r) + \log M(f_n).
\end{align*}
Thus, we obtain the deterministic bound 
\begin{equation} \label{eq:maxzub}
	\max_{|z| = r} \frac{1}{n} \log |f_n(z)| \leq \frac{1}{n} \log M(f_n) + \log(2 + r). 
\end{equation} 
	
For the other direction, we bound 
\begin{align*}
	\min_{|z| = r} \log |f_n(z)| &\geq \sum_{j=1}^n \min_{|z|= r} \log |z - X_j| \\
	&= \sum_{j=1}^n\log |r - |X_j| | \\
	&\geq \sum_{j: |X_j| > r+1 } \log (|X_j| - r) - \sum_{j=1}^n \lp \left| \frac{1}{|X_j| - r}  \right| \\
	&= \sum_{j: |X_j| > r+1 } \log |X_j| + \sum_{j: |X_j| > r+1 } \log \left( 1 - \frac{r}{|X_j|} \right) - \sum_{j=1}^n \lp \left| \frac{1}{|X_j| - r} \right| \\
	&\geq \log M(f_n) - \sum_{j: 1 < |X_j| \leq r+1 } \log |X_j| - n \log(1 + r) - \sum_{j=1}^n \lp \left| \frac{1}{|X_j| - r} \right| \\
	&\geq \log M(f_n) - 2n \log(1 + r) - \sum_{j=1}^n \lp \left| \frac{1}{|X_j| - r} \right|.  
\end{align*}
For the last term on the right-hand side, the law of large numbers implies that almost surely 
\[ \frac{1}{n}\sum_{j=1}^n \lp \left|\frac{1}{|X_j| - r}\right| \longrightarrow \int_{\C} \log_+\left| \frac{1}{|w| - r} \right|\d\mu(w) < \infty, \] 
where we used the assumption that $r \not\in F_0^\mu$.  Together with \eqref{eq:maxzub}, the bounds above complete the proof of the lemma in the case when $x = 0$.  

Fix $x \in \mathbb{C}$ with $x \neq 0$, and define $g_n(z) := f_n(x + z)$.  The roots of $g_n$ are simply the roots of $f_n$ shifted by $x$.  If $\nu$ is the probability measure formed from shifting $\mu$ by $x$, then the roots of $g_n$ are drawn independently from $\nu$.  Since $\mu \in \mathcal{P}(\mathbb{C})$ was arbitrary in the argument above, we can apply that argument to $\nu$ to conclude that, for all $r \in  (0, \infty) \setminus F_0^\nu = (0, \infty) \setminus F_x^\mu$, almost surely
\[ \limsup_{n \to \infty} \frac{1}{n} \max_{|z| = r} | \log |g_n(z)| - \log M(g_n) | < \infty. \]
Thus, in order to complete the proof, it remains to show
\begin{equation} \label{eq:MfnMgn}
	| \log M(f_n) - \log M(g_n) | = O_x(n) 
\end{equation}
almost surely.  Since $x \in \mathbb{C}$ is fixed, the bound in \eqref{eq:MfnMgn} follows from a simple term-by-term comparison using \eqref{eq:def:Mf}.  This completes the proof of the lemma.  
\end{proof}

We will need the following deterministic proposition in order to complete the proof of Theorem \ref{thm:nolog}.  In order to handle cases \eqref{item:infty}, \eqref{item:-infty}, and \eqref{item:both} simultaneously, we assume the two polynomials have degree $d_n$ for an arbitrary degree sequence $\{d_n\}$.  We use the notation
\[ \Q + i \Q = \{a + ib : a, b \in \mathbb{Q} \}, \]
where $i$ denotes the imaginary unit.  We let $\partial B$ be the boundary of $B \subset \mathbb{C}$.  

\begin{proposition}\label{pr:deterministic}
Let $\{\alpha_n\}_{n \geq 1}$ and $\{\beta_n\}_{n \geq 1}$ be sequences of complex numbers, and define the sequence of polynomials
\[ f_n(z) := \prod_{j=1}^{d_n} (z - \alpha_j), \qquad g_n(z) := \prod_{j=1}^{d_n} (z - \beta_j) \]
for a degree sequence $\{d_n\}_{n \geq 1}$.  Suppose the empirical measure
\[ \frac{1}{d_n} \sum_{j=1}^{d_n} \delta_{\alpha_j} \]
constructed from the roots of $f_n$ converges weakly to a measure $\mu \in \mathcal{P}(\mathbb{C})$ as $n \to \infty$.  Assume, for each $x \in \Q + i \Q$, there is a countable dense subset $R_x \subset (0, \infty)$ so that, for all $r \in R_x$, 
 \begin{equation}\label{eq:assumption-circles-f}
	\limsup_{n \to \infty} \frac{1}{d_n} \max_{|z| = r}\left| \log |f_n(x + z)| - \log M(f_n) \right| < \infty  
\end{equation} 
and
\begin{equation}\label{eq:assumption-circles-g}
	\limsup_{n \to \infty} \frac{1}{d_n} \max_{|z| = r}\left| \log |g_n(x + z)| - \log M(g_n) \right| < \infty.  
\end{equation}
In addition, assume for each $x \in \Q + i \Q$ and $r \in R_x$ that $\mu( \partial B(x, r)) = 0$.  
If 
\begin{equation}\label{eq:assumption-mahler}
	\lim_{n \to \infty} \frac{1}{d_n} \left(\log M(f_n) - \log M(g_n) \right)  = + \infty, 
\end{equation}
then the empirical measure 
\[ \rho_n := \frac{1}{d_n} \sum_{j=1}^{d_n} \delta_{z_j^{(n)}} \]
formed from the roots $z_1^{(n)}, \ldots, z_{d_n}^{(n)}$ of $f_n + g_n$ converges weakly to $\mu$ as $n \to \infty$.  
\end{proposition}
\begin{proof}
Let $\mathcal{A}'$ be the collection of all open balls $B(x, r)$, where $x \in \Q + i \Q$ and $r \in R_x$.  Let $\mathcal{A}$ be the collection of all finite intersections of the open balls from $\mathcal{A}'$.  It follows that $\mathcal{A}$ is a $\pi$-system and that any open set in $\mathbb{C}$ can be written as a countable union of sets in $\mathcal{A}$.  Thus, by Theorem 2.2 in \cite{MR1700749}, it suffices to show that $\rho_n(B) \to \mu(B)$ as $n \to \infty$ for any $B \in \mathcal{A}$.  

Fix $B \in \mathcal{A}$.  Then $B = B_1 \cap \cdots \cap B_k$, where $B_1, \ldots, B_k \in \mathcal{A}'$.  By supposition, $\mu(\partial B) = 0$ (i.e., $ B$ is a $\mu$-continuity set).  Thus, since the empirical measure constructed from the roots of $f_n$ converges weakly to $\mu$ by assumption, it follows that
\[ \frac{ | \{ 1 \leq j \leq d_n : \alpha_j \in B \} |}{ d_n } \longrightarrow \mu(B) \]
as $n \to \infty$, where $|S|$ denotes the cardinality of the finite set $S$.  Therefore, to complete the proof, it suffices to show that $f_n$ has the same number of roots in $B$ as $f_n + g_n$ for all $n$ sufficiently large.  For this we will use Rouch\'e's theorem.  Indeed, by suppositions \eqref{eq:assumption-circles-f} and \eqref{eq:assumption-circles-g}, there exists a constant $C > 0$ so that 
\[ \frac{1}{d_n} \log |f_n(z)| \geq \frac{1}{d_n} \log M(f_n) - C \]
and 
\[ \frac{1}{d_n} \log |g_n(z)| \leq \frac{1}{d_n} \log M(g_n) + C \]
for all $z \in \partial B_l$, $1 \leq l \leq k$ and all sufficiently large values of $n$.  In view of \eqref{eq:assumption-mahler} and the fact that $\partial B \subset \cup_{l=1}^k \partial B_l$, we conclude that
\[ |g_n(z)| < |f_n(z)| \]
for all $z \in \partial B$ and all sufficiently large values of $n$.  Therefore, by Rouch\'e's theorem (see Exercise 24 in Chapter 10 of \cite{MR924157} for the general form of Rouch\'e's theorem used here), $f_n$ and $f_n + g_n$ have the same number of roots in $B$ for all $n$ sufficiently large.  This completes the proof of the proposition.  
\end{proof}

We now complete the proof of Theorem \ref{thm:nolog}.  

\begin{proof}[Proof of Theorem \ref{thm:nolog}]
To start, assume case \eqref{item:infty} holds.  Recall that this means the event 
\[ \Omega_1 := \left\{ \lim_{n \to \infty} \frac{1}{n} \left( \log M(p_n) - \log M(q_n) \right) = +\infty \right\} \]
holds with probability $1$.  

In view of Lemmas \ref{lem:Fmur} and \ref{lem:replace-Mahler}, for each $x \in \Q + i \Q$, let $R_x \subset (0, \infty)$ be a countable dense set\footnote{Since $\mu$ and $\nu$ are probability measures, for any fixed $x \in \mathbb{C}$, there is at most a countable number of values for $r > 0$ so that $\mu( \partial B(x,r)) + \nu( \partial B(x,r)) > 0$.} with $\mu( \partial B(x, r)) = 0 = \nu( \partial B(x, r))$ for all $r \in R_x$ and such that 
\[ \limsup_{n\to\infty} \frac{1}{n} \max_{|z| = r} \left| \log |p_n(x + z)| - \log M(p_n) \right| < \infty \]
and
\[ \limsup_{n\to\infty} \frac{1}{n} \max_{|z| = r} \left| \log |q_n(x + z)| - \log M(q_n) \right| < \infty \]
almost surely for $r \in R_x$.  Let $\Omega_2$ be the event that these two limit superiors are finite for all $x \in \Q + i \Q$ and each $r \in R_x$.  Since $\Q + i \Q$ is countable and, for each $x \in \Q + i \Q$, $R_x$ is countable, it follows that $\Omega_2$ holds with probability $1$.  Let $\Omega_3$ be the event that the empirical measure
\[ \frac{1}{n} \sum_{i=1}^n \delta_{X_i} \]
constructed from the roots of $p_n$ converges weakly to $\mu$ as $n \to \infty$.  It follows from Proposition \ref{prop:random-measures} and the law of large numbers that $\Omega_3$ holds with probability $1$, and hence the event $\Omega := \Omega_1 \cap \Omega_2 \cap \Omega_3$ also holds with probability $1$.  
Proposition \ref{pr:deterministic} implies that on $\Omega$, $\rho_n$ converges weakly to $\mu$ as $n \to \infty$.  This completes the proof of case \eqref{item:infty}.  

By reversing the roles of $p_n$ and $q_n$, an identical argument applies when case \eqref{item:-infty} holds; we omit the details.  

Finally, assume case \eqref{item:both} holds.  Recall that this means the event
\[ \Omega_4 := \left\{ \limsup S_n = +\infty \text{ and } \liminf S_n = - \infty \right\} \]
holds with probability $1$, where $S_n$ is defined in \eqref{eq:Sn}.  Let $\Omega_5$ be the event that the empirical measure
\[ \frac{1}{n} \sum_{i=1}^n \delta_{Y_i} \]
constructed from the roots of $q_n$ converges weakly to $\nu$ as $n \to \infty$.  
It follows from Proposition \ref{prop:random-measures} and the law of large numbers that $\Omega_5$ holds with probability $1$. 
Fix a realization $\omega \in \Omega_2 \cap \Omega_3 \cap \Omega_4 \cap \Omega_5$.  Then there exist subsequences $\{n_k\}$ and $\{ n_l\}$ (depending on $\omega$) so that $\lim_{k \to \infty} S_{n_k} = +\infty$ and $\lim_{l \to \infty} S_{n_l} = - \infty$.  Applying Proposition \ref{pr:deterministic} along these subsequences, we conclude that $\rho_{n_k}$ converges weakly to $\mu$ as $k \to \infty$ and $\rho_{n_l}$ converges weakly to $\nu$ as $l \to \infty$.  Since $\Omega_2 \cap \Omega_3 \cap \Omega_4 \cap \Omega_5$ holds with probability $1$, this completes the proof of case  \eqref{item:both}.  
\end{proof}

\appendix
\section{Convergence of random measures}

Recall that $B(r)$ is the open ball of radius $r > 0$ centered at the origin in the complex plane.  The following result concerning the weak convergence of random probability measures can be deduced from the more general results in \cite{MR2236634}; we provide a proof for completeness.  

\begin{proposition} \label{prop:random-measures}
Let $\{\mu_n\}_{n \geq 1}$ be a sequence of random probability measures on $\mathbb{C}$, all defined on the same probability space, and let $\mu \in \mathcal{P}(\mathbb{C})$ be deterministic.  If, for every bounded continuous function $f: \mathbb{C} \to \mathbb{C}$, 
\begin{equation} \label{eq:bndctscon}
	\int_{\mathbb{C}} f \d \mu_n \longrightarrow \int_{\mathbb{C}} f \d \mu 
\end{equation} 
almost surely as $n \to \infty$, then
\begin{equation} \label{eq:convprob1}
	\P \left( \{\mu_n\} \text{ converges weakly to } \mu \text{ as } n \to \infty \right) = 1. 
\end{equation}
\end{proposition}
\begin{proof}
For any $k \in \mathbb{N}$, let $r_k > 0$ be such that $\mu(\mathbb{C} \setminus B(r_k)) < 1/k$.  Let $f_k: \mathbb{C} \to [0,1]$ be a continuous function so that $f_k(z) = 1$ for all $z \in B(r_k)$ and $f_k(z) = 0$ for all $z \notin B(r_k + 1)$.  Let $\Omega_1$ be the event that
\[ \int_{\mathbb{C}} f_k \d \mu_n \longrightarrow \int_{\mathbb{C}} f_k \d \mu \]
for all $k \in \mathbb{N}$.  By supposition \eqref{eq:bndctscon} and the fact that $\mathbb{N}$ is countable, $\Omega_1$ holds with probability $1$.  

In addition, on $\Omega_1$, the sequence $\{\mu_n\}_{n \geq 1}$ is tight.  Indeed, given any $\eps > 0$, there exists $1/k < \eps$ so that 
\[ \mu_n(B(r_k+1)) \geq \int_{\mathbb{C}} f_k \d \mu_n \longrightarrow \int_{\mathbb{C}} f_k \d \mu > \mu(B(r_k)) \geq 1- 1/k > 1 - \eps \]
on $\Omega_1$.  

Define the (random) characteristic function of $\mu_n$ by
\[ \phi_n(s,t) := \int_{\mathbb{C}} e^{i (s \Re(z) + t \Im(z))} \d \mu_n(z), \qquad s,t \in \mathbb{R} \]
and the deterministic characteristic function of $\mu$ by
\[ \phi(s,t) := \int_{\mathbb{C}} e^{i (s \Re(z) + t \Im(z))} \d \mu(z), \qquad s,t \in \mathbb{R}. \]
Let $\Omega_2$ be the event that $\phi_n(s,t) \to \phi(s,t)$ for all $s, t \in \mathbb{Q}$.  It follows from assumption \eqref{eq:bndctscon} that $\Omega_2$ holds with probability $1$, and hence the event $\Omega := \Omega_1 \cap \Omega_2$ holds with probability $1$.  

Fix a realization $\omega \in \Omega$.  We will show that $\{ \mu_n \}_{n \geq 1}$ converges weakly to $\mu$ for this fixed realization $\omega$.  Let $\{\mu_{n_k}\}_{k \geq 1}$ be an arbitrary subsequence.  Since $\{ \mu_n \}_{n \geq 1}$ is tight, $\{\mu_{n_k}\}_{k \geq 1}$ is also tight and there exists a further subsequence $\{ \mu_{n_{k_l}} \}_{l \geq 1}$ and a measure $\mu' \in \mathcal{P}(\mathbb{C})$ so that $\{ \mu_{n_{k_l}} \}_{l \geq 1}$ converges weakly to $\mu'$ as $l \to \infty$.  This implies that $\phi_{n_{k_l}} (s,t) \to \phi'(s,t)$ as $l \to \infty$ for all $s,t \in \mathbb{R}$, where 
\[ \phi'(s,t) := \int_{\mathbb{C}} e^{i (s \Re(z) + t \Im(z))} \d \mu'(z), \qquad s,t \in \mathbb{R} \]
is the characteristic function of $\mu'$.  However, it follows that $\phi_{n_{k_l}} (s,t) \to \phi(s,t)$ as $l \to \infty$ for all $s, t \in \mathbb{Q}$.  By continuity of the characteristic functions $\phi$ and $\phi'$, it must be the case that $\phi(s,t) = \phi'(s,t)$ for all $s,t \in \mathbb{R}$, and hence $\mu = \mu'$.  

Therefore, we have shown that every subsequence of $\{ \mu_n \}_{n \geq 1}$ has a further subsequence that converges weakly to $\mu$.  It follows from Theorem 2.6 in \cite{MR1700749} that $\{ \mu_n \}_{n \geq 1}$ converges weakly to $\mu$ as $n \to \infty$ for any fixed realization $\omega \in \Omega$.  Since $\Omega$ holds with probability $1$, the conclusion in \eqref{eq:convprob1} follows.  
\end{proof}

\bibliographystyle{abbrv}
\bibliography{library}

\begin{thebibliography}{10}

\bibitem{MR2236634}
P.~Berti, L.~Pratelli, and P.~Rigo.
\newblock Almost sure weak convergence of random probability measures.
\newblock {\em Stochastics}, 78(2):91--97, 2006.

\bibitem{MR1700749}
P.~Billingsley.
\newblock {\em Convergence of probability measures}.
\newblock Wiley Series in Probability and Statistics: Probability and
  Statistics. John Wiley \& Sons, Inc., New York, second edition, 1999.
\newblock A Wiley-Interscience Publication.

\bibitem{Bordenave}
C.~Bordenave and D.~Chafa\"{\i}.
\newblock Around the circular law.
\newblock {\em Probability Surveys}, 9:1--89, 2012.

\bibitem{1801.08974}
S.-S. Byun, J.~Lee, and T.~R. Reddy.
\newblock Zeros of random polynomials and its higher derivatives.
\newblock Available at arXiv:1801.08974, 2018.

\bibitem{MR3196447}
F.~Calogero.
\newblock Properties of the zeros of the sum of two polynomials.
\newblock {\em J. Nonlinear Math. Phys.}, 20(3):348--354, 2013.

\bibitem{MR3318313}
P.-L. Cheung, T.~W. Ng, J.~Tsai, and S.~C.~P. Yam.
\newblock Higher-order, polar and {S}z.-{N}agy's generalized derivatives of
  random polynomials with independent and identically distributed zeros on the
  unit circle.
\newblock {\em Comput. Methods Funct. Theory}, 15(1):159--186, 2015.

\bibitem{MR336806}
K.~B. Erickson.
\newblock The strong law of large numbers when the mean is undefined.
\newblock {\em Trans. Amer. Math. Soc.}, 185:371--381 (1974), 1973.

\bibitem{math/0612833}
S.~Fisk.
\newblock Polynomials, roots, and interlacing.
\newblock Available at arXiv:math/0612833, 2006.

\bibitem{MR2339369}
U.~Haagerup and H.~Schultz.
\newblock Brown measures of unbounded operators affiliated with a finite von
  {N}eumann algebra.
\newblock {\em Math. Scand.}, 100(2):209--263, 2007.

\bibitem{MR3689975}
B.~Hanin.
\newblock Pairing of zeros and critical points for random polynomials.
\newblock {\em Ann. Inst. Henri Poincar\'{e} Probab. Stat.}, 53(3):1498--1511,
  2017.

\bibitem{MR1049148}
W.~K. Hayman.
\newblock {\em Subharmonic functions. {V}ol. 2}, volume~20 of {\em London
  Mathematical Society Monographs}.
\newblock Academic Press, Inc. [Harcourt Brace Jovanovich, Publishers], London,
  1989.

\bibitem{MR0460672}
W.~K. Hayman and P.~B. Kennedy.
\newblock {\em Subharmonic functions. {V}ol. {I}}.
\newblock London Mathematical Society Monographs, No. 9. Academic Press
  [Harcourt Brace Jovanovich, Publishers], London-New York, 1976.

\bibitem{HKYV}
J.~B. Hough, M.~Krishnapur, Y.~Peres, and B.~Vir\'{a}g.
\newblock {\em Zeros of {G}aussian analytic functions and determinantal point
  processes}, volume~51 of {\em University Lecture Series}.
\newblock American Mathematical Society, Providence, RI, 2009.

\bibitem{MR3283656}
Z.~Kabluchko.
\newblock Critical points of random polynomials with independent identically
  distributed roots.
\newblock {\em Proc. Amer. Math. Soc.}, 143(2):695--702, 2015.

\bibitem{MR3940764}
Z.~Kabluchko and H.~Seidel.
\newblock Distances between zeroes and critical points for random polynomials
  with i.i.d. zeroes.
\newblock {\em Electron. J. Probab.}, 24:Paper No. 34, 25, 2019.

\bibitem{KZ}
Z.~Kabluchko and D.~Zaporozhets.
\newblock Asymptotic distribution of complex zeros of random analytic
  functions.
\newblock {\em The Annals of Probability}, 42(4):1374--1395, 2014.

\bibitem{MR266315}
H.~Kesten.
\newblock The limit points of a normalized random walk.
\newblock {\em Ann. Math. Statist.}, 41:1173--1205, 1970.

\bibitem{MR1923392}
S.-H. Kim.
\newblock Factorization of sums of polynomials.
\newblock {\em Acta Appl. Math.}, 73(3):275--284, 2002.

\bibitem{MR1908370}
S.-H. Kim.
\newblock Sums of two polynomials with each having real zeros symmetric with
  the other.
\newblock {\em Proc. Indian Acad. Sci. Math. Sci.}, 112(2):283--288, 2002.

\bibitem{Marden}
M.~Marden.
\newblock {\em Geometry of polynomials}.
\newblock Mathematical Surveys, No. 3. American Mathematical Society,
  Providence, R.I., second edition, 1966.

\bibitem{MR0181738}
A.~I. Markushevich.
\newblock {\em Theory of functions of a complex variable. {V}ol. {II}}.
\newblock Prentice-Hall, Inc., Englewood Cliffs, N.J., 1965.
\newblock Revised English edition translated and edited by Richard A.
  Silverman.

\bibitem{ORourke}
S.~O'Rourke and T.~R. Reddy.
\newblock Sums of random polynomials with independent roots.
\newblock {\em Journal of Mathematical Analysis and Applications},
  495(1):124719, 23, 2021.

\bibitem{MR3896083}
S.~O'Rourke and N.~Williams.
\newblock Pairing between zeros and critical points of random polynomials with
  independent roots.
\newblock {\em Trans. Amer. Math. Soc.}, 371(4):2343--2381, 2019.

\bibitem{MR4136480}
S.~O'Rourke and N.~Williams.
\newblock On the local pairing behavior of critical points and roots of random
  polynomials.
\newblock {\em Electron. J. Probab.}, 25:Paper No. 100, 68, 2020.

\bibitem{MR3363974}
R.~Pemantle and I.~Rivin.
\newblock The distribution of zeros of the derivative of a random polynomial.
\newblock In {\em Advances in combinatorics}, pages 259--273. Springer,
  Heidelberg, 2013.

\bibitem{Petrov}
V.~V. Petrov.
\newblock {\em Limit theorems of probability theory}, volume~4 of {\em Oxford
  Studies in Probability}.
\newblock The Clarendon Press, Oxford University Press, New York, 1995.
\newblock Sequences of independent random variables, Oxford Science
  Publications.

\bibitem{MR1911767}
A.~Pint\'{e}r.
\newblock Zeros of the sum of polynomials.
\newblock {\em J. Math. Anal. Appl.}, 270(1):303--305, 2002.

\bibitem{Ransford}
T.~Ransford.
\newblock {\em Potential theory in the complex plane}, volume~28 of {\em London
  Mathematical Society Student Texts}.
\newblock Cambridge University Press, Cambridge, 1995.

\bibitem{MR3698743}
T.~R. Reddy.
\newblock Limiting empirical distribution of zeros and critical points of
  random polynomials agree in general.
\newblock {\em Electron. J. Probab.}, 22:Paper No. 74, 18, 2017.

\bibitem{MR924157}
W.~Rudin.
\newblock {\em Real and complex analysis}.
\newblock McGraw-Hill Book Co., New York, third edition, 1987.

\bibitem{Saff}
E.~B. Saff and V.~Totik.
\newblock {\em Logarithmic potentials with external fields}, volume 316 of {\em
  Grundlehren der Mathematischen Wissenschaften [Fundamental Principles of
  Mathematical Sciences]}.
\newblock Springer-Verlag, Berlin, 1997.
\newblock Appendix B by Thomas Bloom.

\bibitem{MR1188001}
M.~Sodin.
\newblock Value distribution of sequences of rational functions.
\newblock In {\em Entire and subharmonic functions}, volume~11 of {\em Adv.
  Soviet Math.}, pages 7--20. Amer. Math. Soc., Providence, RI, 1992.

\bibitem{MR2970701}
S.~D. Subramanian.
\newblock On the distribution of critical points of a polynomial.
\newblock {\em Electron. Commun. Probab.}, 17:no. 37, 9, 2012.

\bibitem{TVcirc}
T.~Tao and V.~Vu.
\newblock Random matrices: universality of {ESD}s and the circular law.
\newblock {\em The Annals of Probability}, 38(5):2023--2065, 2010.
\newblock With an appendix by Manjunath Krishnapur.

\bibitem{MR202981}
R.~Vermes.
\newblock On the zeros of a linear combination of polynomials.
\newblock {\em Pacific J. Math.}, 19:553--559, 1966.

\bibitem{Walsh}
J.~L. Walsh.
\newblock On the location of the roots of certain types of polynomials.
\newblock {\em Trans. Amer. Math. Soc.}, 24(3):163--180, 1922.

\bibitem{MR171902}
M.~Zedek.
\newblock Continuity and location of zeros of linear combinations of
  polynomials.
\newblock {\em Proc. Amer. Math. Soc.}, 16:78--84, 1965.

\end{thebibliography}

\end{document}